\documentclass[12pt]{amsart}

\usepackage{amsmath,amssymb,amscd,amsthm,amsxtra}
\usepackage{latexsym}
\usepackage{color}
\allowdisplaybreaks[4]

\newtheorem{theorem}{Theorem}
\newtheorem{theoremA}{Theorem}

\newtheorem{lemma}{Lemma}

\newtheorem{remark}{Remark}

\newtheorem{corollary}{Corollary}

\newcommand{\q}{\quad}
\newcommand{\qq}{\quad\quad}

\newcommand{\nf}{\infty}

\newcommand{\al}{\alpha}
\newcommand{\be}{\beta}
\newcommand{\ga}{\gamma}
\newcommand{\Ga}{\Gamma}
\newcommand{\de}{\delta}

\newcommand{\De}{\Delta}

\newcommand{\la}{\lambda}

\newcommand{\si}{\sigma}

\newcommand{\vp}{\Psi}

\newcommand{\rn}{{\mathbf R}^n}

\newcommand{\zn}{\mathbf Z^n}

\newcommand{\zzz}{\mathbf Z}

\newcommand{\ccc}{\mathbf C}

\newcommand{\lprn}{L^{p}(\rn)}

\newcommand{\cs}{\mathcal S}

\newcommand{\lab}{\label}

\newcommand{\intrn}{\int_{\rn}}

\newcommand{\f}{\frac}

\newcommand{\p}{\partial}

\newcommand{\abs}[1]{\left|#1\right|}
\newcommand{\n}[2]{{\left\| #1 \right\|}_{#2}}
\newcommand{\lan}[1]{\left\langle #1\right\rangle}
\newcommand{\wh}[1]{\widehat{#1}}
\newcommand{\wt}[1]{\widetilde{#1}}
\newcommand{\na}{\nabla}
\newcommand{\wtv}{\wt{\vp}}

\newcommand{\br}{{\mathbf R}}
 
\newcommand{\bz}{\mathbf Z}
\newcommand{\bn}{\mathbf N}

\newcommand{\cf}{\mathcal F}
\newcommand{\RE}{\textnormal{Re}\,}
\begin{document}

\title
[Kato-Ponce]
{The Kato-Ponce inequality}

\author{Loukas Grafakos, Seungly Oh}

\address{
Department of Mathematics\\
University of Missouri\\
Columbia, MO 65211, USA}

\email{grafakosl@missouri.edu}

\email{ohseun@missouri.edu}

\thanks{Grafakos' research partially
supported by the NSF under grant DMS 0900946}

\date{\today}

\subjclass{Primary 42B20. Secondary  46E35}

\keywords{Kato-Ponce}

\begin{abstract} 
In this article we revisit the inequalities of Kato and Ponce concerning the $L^r$ norm of the Bessel potential $J^s=(1-\De)^{s/2}$ (or Riesz potential $D^s = (-\De)^{s/2}$)
 of the  product of two functions in terms of the product of the $L^{p}$ norm  of one function and the $L^{q}$ norm of the the Bessel potential $J^s$ (resp. Riesz potential $D^s$)
 of the other function. Here the indices $p$, $q$, and $r$ are related as in H\"older's inequality $1/p+1/q=1/r$ and they satisfy $1\leq p,q \leq \infty$ and $1/2\leq r<\infty$. Also the estimate is weak-type in the  case when either $p$ or $q$ is equal to $1$. In the case when $r<1$ we indicate via an example that when $s\le n/r-n$ the inequality  fails.  Furthermore, we extend these results to the multi-parameter case.
\end{abstract}

\maketitle

\section{Introduction}
In \cite{KP}, Kato and Ponce obtained the commutator estimate
\[
\n{J^{s} (fg)- f(J^s g)}{L^p(\rn)} \leq C\big[\n{\na f}{L^{\infty}(\rn)} \n{J^{s-1} g}{L^p(\rn)} + \n{J^s f}{L^p(\rn)} \n{g}{L^{\infty}(\rn)}\big]
\]
for $1<p<\infty$ and $s>0$, where $J^s := (1-\De)^{s/2}$ is the Bessel potential,  $\na$ is the $n$-dimensional gradient, $f$, $g$ are Schwartz functions, and $C$ is a constant depending on $n$, $p$ and $s$.  This estimate was obtained by applying the Coifman-Meyer multiplier theorem~\cite{CM} and Stein's complex interpolation theorem for analytic families~\cite{St}. {Using a homogeneous symbol $D^s := (-\De)^{s/2}$ instead, Kenig, Ponce, Vega \cite{KPV} obtained the following estimate,
\[
\n{D^s [fg] - f D^s f - g D^s f}{L^r} \leq C(s,s_1, s_2, r,p,q) \n{D^{s_1} f}{L^p} \n{D^{s_2} g}{L^q}
\]
where $s= s_1 + s_2$ for $s, s_1, s_2 \in (0,1)$, and $1<p,q,r <\infty$ such that $\f{1}{r} = \f{1}{p} +\f{1}{r}$. } 

In place of the original statement given by Kato and Ponce, the following variant is  known in the literature as the Kato-Ponce inequality (also \emph{fractional Leibniz rule}) 
\[
\n{J^{s} (fg)}{L^r(\rn)} \leq C \big[\n{f}{L^{p_1}(\rn)} \n{J^{s} g}{L^{q_1}(\rn)} + \n{J^s f}{L^{p_2}(\rn)} \n{g}{L^{q_2}(\rn)}\big]
\]
where $s>0$ and $\f{1}{r} = \f{1}{p_1} + \f{1}{q_1} = \f{1}{p_2} + \f{1}{q_2}$ for $1< r<\infty$, $1< p_1, p_2, q_1, q_2\leq \infty$ and $C=C(s,n,r,p_1,p_2,q_1,q_2)$.  This can be proved in a similar manner \cite{G2, GK}, using the Coifman-Meyer multiplier theorem in conjunction with Stein's complex interpolation.  The homogeneous version of above estimate, where $J^s$ is replaced by $D^s$, can also be proved by  the same approach.   Alternatively, other known proofs involve applications of the Hardy-Littlewood vector-valued  maximal function inequality~\cite{CW}, Mihlin theorem~\cite{MS} or vector-valued Calderon-Zygmund theorem~\cite{BB}, but these methods naturally do not extend to $r<1$.  The approach in this work does not only provide a new proof for the well-studied case $r \geq 1$, but it  also has a natural extension to $r<1$.

There are further generalizations of Kato-Ponce-type inequalities.  For instance, Muscalu, Pipher, Tao, and Thiele, \cite{Acta} extended this inequality to allow for partial fractional derivatives in $\mathbf{R}^2$.  Bernicot, Maldonado, Moen, and Naibo  \cite{BMMN} proved the Kato-Ponce inequality in weighted Lebesgue spaces under certain restrictions on the weights.  The last authors also extended the Kato-Ponce inequality to indices $r<1$ under the assumption $s>n$.

In this article, we prove Kato-Ponce inequality for $1/2\leq r <\infty$,  where Lebesgue spaces~$L^r$ are replaced by weak type Lorentz spaces~$L^{r,\infty}$ when either $p$ or $q$ on the right hand side equals $1$.  This agrees with the range  of indices given by Coifman-Meyer multiplier theorem (Theorem~\ref{th:CM} below), except for the missing endpoint bound $L^1\times L^{\infty} \to L^{1,\infty}$.   When $r>1$ the inequality is valid for all $s\ge 0$ but when $r< 1$ there is a restriction $s>n/r-n$. Moreover, we show via an example that the inequality fails for   $s\le n/r-n$   indicating the sharpness of the restriction. Our inequality  extends  that      of Bernicot, Maldonado, Moen and Naibo   in \cite{BMMN} which requires $s>n$ for $r<1$.    Additionally, we   generalize  the multi-parameter extension of the Kato-Ponce inequality  by Muscalu, Pipher, Tao, and Thiele   \cite{Acta} to allow for partial fractional derivatives in $\br^n$.

We now state two of our main results in this work.  Additional results are obtained in Section~\ref{sec:multiKP}.

\begin{theorem}\label{th:KP1}
Let $\f{1}{2}<r<\infty$, $1<p_1, p_2, q_1, q_2 \leq \infty$ satisfy $\f{1}{r} = \f{1}{p_1}+ \f{1}{q_1} = \f{1}{p_2} + \f{1}{q_2}$. Given  $s>\max \left(0, f{n}{r}-n\right)$ or $s\in 2\bn$, there exists a constant $C= C(n,s,r,p_1,q_1,p_2,q_2)<\infty$ such that for for all $ f,g \in \cs(\rn)$ we have
\begin{align}
\n{D^s (fg)}{L^r(\rn)} \leq C \left[\n{D^s f}{L^{p_1}(\rn)} \n{g}{L^{q_1}(\rn)} + \n{f}{L^{p_2}(\rn)} \n{D^s g}{L^{q_2}(\rn)}\right], \label{eq:kp}\\
\n{J^{s} (fg)}{L^r(\rn)} \leq C(s,n) \left[\n{ f}{L^{p_1}(\rn)} \n{J^{s} g}{L^{q_1}(\rn)} + \n{J^s f}{L^{p_2}(\rn)} \n{g}{L^{q_2}(\rn)}\right] \label{eq:KatoPonce}
\end{align}

Moreover for $r<1$, if one of the indices $p_1$, $p_2$, $q_1$, $q_2$ is equal to $1$, then \eqref{eq:kp} and \eqref{eq:KatoPonce} hold when  the $L^r(\rn)$ norms on the left hand side of  the inequalities are replaced by the $L^{r,\infty}(\rn)$ quasi-norm. \end{theorem}

 We remark that the statement above does not include the endpoint $L^1\times L^{\infty} \to L^{1,\infty}$.  Next, we have a companion theorem that focuses on  negative results which highlight  the sharpness of Theorem~\ref{th:KP1}.

\begin{theorem}\label{th:KP2}
Let $s\leq \max(\f{n}{r}-n,0)$ and $s\not\in 2\bn \cup \{0\}$.  Then both \eqref{eq:kp} and \eqref{eq:KatoPonce} fail  for any $1<p_1,q_1,p_2,q_2<\infty$. 
\end{theorem}

 \emph{Upon completion of this manuscript, the authors discovered that Muscalu and Schlag \cite{MS2} had independently reached   conclusion \eqref{eq:kp} via a different approach, based on discretized paraproducts.  A version of Lemma~\ref{le:homKP} also appears in their text, but otherwise their approach is different from ours.}

The paper is organized as follows.  In Section~\ref{sec:prelim}, we introduce Littlewood-Paley operators and prove an estimate concerning them.  In Section~\ref{sec:homKP}, we  prove Theorems~\ref{th:KP1} and \ref{th:KP2} for the homogeneous version of inequality \eqref{eq:kp}.  In Section~\ref{sec:inhKP}, we   prove Theorems~\ref{th:KP1} and \ref{th:KP2} for the inhomogeneous version of inequality \eqref{eq:KatoPonce}.  In Section~\ref{sec:multiKP}, we   state and prove a multi-parameter generalization of Theorems~\ref{th:KP1} and \ref{th:KP2}.  

\section{Preliminaries}\label{sec:prelim}
We denote by $\langle x,y\rangle $ the inner product in $\br^n$. 
We use the notation $\Psi_t(x) = t^{-n} \Psi(x/t)$ when $t>0$ and $x\in \br^n$. 
We denote by $\mathcal S(\br^n)$ the space of 
all rapidly decreasing functions on $\br^n$ called Schwartz functions. We denote by 
$$
\wh{f}(\xi) =  \cf(f)(x)= \int_{\br^n} f(x) e^{-2\pi i \langle x,\xi\rangle }dx
$$
  the Fourier transform of a Schwartz function $f$ on $\br^n$. We also denote by 
  $\cf^{-1}(f)(x) = \cf(f)(-x)$ the inverse Fourier transform. We recall the classical multiplier result of 
Coifman and Meyer:

\begin{theoremA}[Coifman-Meyer] \label{th:CM}  
Let $m \in L^{\infty}(\br^{2n})$ be smooth away from the origin and satisfy
\[
|\p_{\xi}^{\al} \p_{\eta}^{\be} m| (\xi,\eta) \leq C(\al, \be) (|\xi|+ |\eta|)^{-|\al| - |\be|}
\]
for all $\xi, \eta \in \br\setminus\{0\}$ and $\al,\be \in \bz^n$ multi-indices with $|\al|, |\be| \le 2n+1$.  Then for all $f,g\in \cs(\rn)$, 
\[
\n{\int_{\br^{2n}} m(\xi,\eta) \wh{f}(\xi)\wh{g}(\eta) e^{i\lan{\xi+\eta, \cdot}} \, d\xi\,d\eta}{L^r(\rn)} \leq C(p,q,r, m) \n{f}{L^p(\rn)}\n{g}{L^q(\rn)}
\]
where $\f{1}{2}<r<\infty$, $1<p,q\leq \infty$ satisfy $\f{1}{r} = \f{1}{p}+\f{1}{q}$.  Furthermore, when either
$p$ or $q$ is equal to $1$, then the $L^r(\rn)$~norm on left hand side can be replaced by the $L^{r,\infty}(\rn)$~norm.
\end{theoremA}

We recall the following version of the Littlewood-Paley theorem. 

\begin{theorem}\label{th:LP}
Suppose that $\wh{\Psi}$ is an integrable   function on $\rn$  that satisfies 
\begin{equation}\lab{5.1.3-1aa}
 \sum_{j\in \zzz}  |\Psi(2^{-j} \xi)|^2 \le B^2
\end{equation} 
and 
\begin{equation}\lab{5.1.3-1a}
\sup_{y\in \rn\setminus\{0\}}\sum_{j\in \zzz}  
\int_{|x|\ge 2|y|}\big|\wh{\Psi}_{ 2^{-j}}(x-y) -\wh{\Psi}_{ 2^{-j}}( x)\big| dx \le B
\end{equation} 
Then  there exists  a constant  $  C_{n } <\nf $ 
such that for all   $1< p<\nf$ and all $f$ in $\lprn$,
\begin{equation}\lab{5.1.4-1}
 \Big\|\Big(\sum_{j\in \zzz} |\De_j
(f)|^2\Big)^{\f12}\Big\|_{\lprn} \le C_n B\max\big(p,(p-1)^{-1}\big)  
\big\|f\big\|_{\lprn}
\end{equation}
where $\De_j f := \wh{\Psi}_{2^{-j}} * f$.  There also exists a $ C_{n }'<\nf$  
such that for all $f$ in $ L^1(\rn)$,
\begin{equation}\lab{5.1.4-1-1}
 \Big\|\Big(\sum_{j\in \zzz} |\De_j
(f)|^2\Big)^{\f12}\Big\|_{L^{1,\nf}(\rn)} \le C_{n}'B  \big\|f\big\|_{L^1(\rn)}.
\end{equation}
\end{theorem}

\begin{proof}
We make a few remarks about the proof.  
Clearly the required estimate holds when $p=2$ in view of \eqref{5.1.3-1aa}. To obtain estimate 
\eqref{5.1.4-1-1} and thus the case $p\neq 2$, we 
define an operator $\vec T$  acting on functions on 
$\rn$ as follows:
$$
 \vec T (f)(x) =\{  \De_j (f)(x)\}_j\, . 
$$
The inequalities   (\ref{5.1.4-1}) and   (\ref{5.1.4-1-1})          
 we wish to prove  say simply   that 
$\vec T$ is a bounded operator from $L^p(\rn, \ccc)$ to 
$L^p(\rn , \ell^2)$ and from $L^1(\rn, \ccc)$ to 
$L^{1,\nf}(\rn , \ell^2)$. 
We just proved  that this statement is true when $p=2$, and
therefore  the first  hypothesis of Theorem 4.6.1 in \cite{G1} is
satisfied.  
We now observe that the operator $\vec T$ 
can be written in the  form 
$$
 \vec T (f)(x)=   
\bigg\{ \intrn \wh{\Psi}_{2^{-j}}(x-y) f(y)\, dy  \bigg\}_j=
\intrn \vec K(x-y)(f(y))\, dy,
$$
where for each $x\in \rn$, $\vec K(x)$ is a bounded linear 
operator from $\ccc$ to $\ell^2$ given by 
\begin{equation}\lab{5.1.10-00-00}
\vec K(x)(a)= \{\wh{\Psi}_{2^{-j}}(x)a\}_j.
\end{equation}
We clearly have that $\big\|\vec K(x)\big\|_{\ccc\to \ell^2} =
\big(\sum_j |\wh{\Psi}_{2^{-j}} (x)|^2\big)^{\f12}$, and to be able to apply 
Theorem  4.6.1 in \cite{G1} we  need to know that 
\begin{equation}\lab{5.1.10}
\int_{|x|\ge 2|y|} \big\|\vec K(x-y)-\vec K(x)\big\|_{\ccc\to \ell^2}\, dx \le C_n B,
\qq  y\neq 0.
\end{equation}
 
We clearly have  
\begin{align*}
\big\|\vec K(x-y)-\vec K(x)\big\|_{\ccc\to \ell^2} 
\,\,=&\,\, \bigg(\sum_{j\in \zzz}\big|\wh{\Psi}_{ 2^{-j}}(x-y) -\wh{\Psi}_{ 2^{-j}}( x)\big|^2 \bigg)^{\!\f12} \\  
\,\,\le &\,\, \sum_{j\in \zzz}  
\big|\wh{\Psi}_{ 2^{-j}}(x-y) -\wh{\Psi}_{ 2^{-j}}( x)\big|  \end{align*}
and so condition \eqref{5.1.3-1a} implies \eqref{5.1.10}.
\end{proof}

\begin{corollary}\label{cor:m}
Let $m\in \zn\setminus\{0\}$ and $\wh{\Psi}(x) = \wh{\psi}(x+m)$ for some Schwartz function $\psi$ supported in the annulus $1/2\le |\xi| \le 2$. Then for all $1<p<\infty$, 
\begin{equation}
 \Big\|\Big(\sum_{j\in \zzz} |\De_j
(f)|^2\Big)^{\f12}\Big\|_{\lprn} \le C_n \ln (1+|m|)\max\big(p,(p-1)^{-1}\big)  
\big\|f\big\|_{\lprn}. 
\end{equation}
There also exists $C_n<\infty$ such that for all $f\in L^1(\rn)$,
\begin{equation}
 \Big\|\Big(\sum_{j\in \zzz} |\De_j
(f)|^2\Big)^{\f12}\Big\|_{L^{1,\infty}(\rn)} \le C_n \ln (1+|m|)
\big\|f\big\|_{L^1(\rn)}. 
\end{equation}
\end{corollary}
\begin{proof}
Note
$$
\Psi(\xi) = \psi(\xi) e^{2\pi i \lan{m, \xi}}.
$$
The fact that $\Psi$ is supported  
in the annulus $1/2\le |\xi| \le 2$ implies condition \eqref{5.1.3-1aa} for $\Psi $. 
We now focus on condition \eqref{5.1.3-1a} for $\wh{\Psi}$. 

We fix a nonzero $y$ in $\rn$ and $j\in \zzz$. We look at 
$$
\int_{|x|\ge 2|y|}\big|\wh{\Psi}_{ 2^{-j}}(x-y) -\wh{\Psi}_{ 2^{-j}}( x)\big| dx
=\int_{|x|\ge 2|y|}2^{jn} \big|\wh{\psi}( 2^jx-2^jy+m) -\wh{\psi}(2^jx +m)  \big| dx
$$
Changing variables we can write the above as 
$$
I_j=\int_{|x|\ge 2|y|}\big|\wh{\Psi}_{ 2^{-j}}(x-y) -\wh{\Psi}_{ 2^{-j}}( x)\big| dx
=\int_{|x-m|\ge 2^{j+1}|y|}  \big|\wh{\psi}( x-2^jy ) -\wh{\psi}(x)  \big| dx
$$

\textbf{Case 1:}  $2^j \ge 2\, |m|\, |y|^{-1}$. In this case we estimate $I_j$ by 
$$
\int_{|x-m|\ge 2^{j+1}|y|} \f{c}{(1+|x-2^jy |)^{n+2}} dx+ \int_{|x-m|\ge 2^{j+1}|y|}\f{c}{(1+|x  |)^{n+2}} \, dx
$$
$$=
\int_{|x+2^jy-m|\ge 2^{j+1}|y|} \f{c}{(1+|x |)^{n+2}} dx+ \int_{|x-m|\ge 2^{j+1}|y|}\f{c}{(1+|x  |)^{n+2}} \, dx
$$
Suppose that $x$ lies in the domain of integration of the first integral. Then 
$$
|x|\ge |x+2^jy-m|-2^j|y|-|m|\ge 2^{j+1}|y|-2^j|y|- \f12 \, 2^j|y| = \f12 \, 2^j|y|.
$$
If $x$ lies in the domain of integration of the second integral, then 
$$
|x|\ge |x -m| -|m|\ge  2^{j+1}|y|-|m| \ge 2^{j+1}|y| - \f12 \, 2^j|y| =  \f32 \, 2^j|y|.
$$
In both cases we have 
$$
I_j \le 2 \int_{|x | \ge \f12 \, 2^j  |y|}\f{c}{(1+|x  |)^{n+2}} \, dx \le \f{C}{2^j |y|}
 \int_{\rn}\f{1}{(1+ |x |)^{n+1}} \, dx \le \f{C_n}{2^j|y|},
$$
and clearly 
$$
\sum_{j:\,\, 2^j|y|\ge 2|m|} I_j \le \sum_{j:\,\, 2^j|y|\ge 2 } I_j \le C_n.
$$

\textbf{Case 2:}  $|y|^{-1}\le 2^j \le 2\, |m|\, |y|^{-1}$. The number of $j$'s in this case are $O(\ln |m|)$.  Thus, uniformly bounding $I_j$ by a constant, we obtain
\[
\sum_{j: 1 \le 2^j|y| \le 2|m|} I_j \leq C_n (1+\ln |m|).
\]

\textbf{Case 3.} $2^j\le |y|^{-1}$. In this case we have 
$$
\big| \psi( x-2^jy ) -\psi(x) \big| =\left| \int_0^1 2^j \lan{\nabla \psi (x-2^j t y), y}\, dt\right| \le 
2^j |y| \int_0^1 \f{c}{(1+|x-2^j t y|)^{n+1}}\, dt\, . 
$$
Integrating over $x \in \rn$ gives the bound  $I_j \le C_n\, 2^j |y|$.  Thus, we obtain
\[
\sum_{j: 2^j|y| \le 1} I_j \le C_n.
\]

Overall, we obtain the bound $C_n \ln (1+|m|)$ for \eqref{5.1.3-1a}, which yields the desired statement by Theorem~\ref{th:LP}.
\end{proof}

\section{Homogeneous Kato-Ponce inequality}\label{sec:homKP}

In the following lemma, we recall the explicit formula for the Riesz potential described in \cite{G1}.
\begin{lemma}\label{le:homKP}
Let $f\in \cs(\rn)$ be fixed.  Given $s>0$, define $f_s := D^s f$.  Then $f_s$ lies in  $ L^{\infty} (\rn)$ and satisfies the following asymptotic estimate:
\begin{itemize}
\item  There exists a constant $C(n,s,f)$ such that
\begin{equation}\label{eq:est1}
|f_s(x)| \leq C(n,s,f) |x|^{-n-s}\qq \forall x:\, |x|>1.
\end{equation}
\item Let $s\not\in 2\bn$.  If $f(x)\geq 0$  for $\forall x\in \rn$ and $f \not\equiv 0$, then there exists $R\gg 1$ and a constant~$C(n,s,f,R)$ such that
\begin{equation}\label{eq:est2}
|f_s(x)| \geq C(n,s,f,R) |x|^{-n-s}\qq  \forall x:\,|x|>R. 
\end{equation}
\end{itemize}
\end{lemma}

\begin{proof}
For any $z\in \ccc$ with $\textnormal{Re}\, z >-n$ and $g\in \cs(\rn)$, define the distribution $u_z$ by
\begin{equation}\label{def:uz}
\lan{u_z, g} := \int_{\rn} \f{\pi^{\f{z+n}{2}}}{\Ga \left(\f{z+n}{2}\right)} |x|^z g(x)\, dx,
\end{equation}
where $\Ga(\cdot)$ denotes the gamma function.  We recall Theorem~2.4.6 of \cite{G1} and the preceding remarks:
\begin{itemize}
\item For any $g\in \cs(\rn)$, the map $z\mapsto \lan{u_z, g}$ on the half-plane $\textnormal{Re}\, z> -n$ has an holomorphic extension to the entire complex plane.

\item $\lan{u_z, \wh{g}\,} = \lan{u_{-n-z}, \wh{g}\,}$, where $\lan{u_z, f}$ is understood as the holomorphic extension when $\textnormal{Re}\, z\leq -n$.

\item Both $u_z,\, u_{-n-z} \in L^1_{\textnormal{loc}}$ if and only if $-n< \textnormal{Re}\, z <0$, in which case both $u_z$ and $\wh{u_z}$ are well-defined by \eqref{def:uz}.
\end{itemize}

Now, fix $f\in \cs(\rn)$ note that for $s>0$,
\[
f_s (x) = \int_{\rn} |\xi|^s \wh{f}(\xi) \, e^{2\pi i \lan{\xi, x}}\, d\xi = \f{\Ga \left(\f{s+n}{2}\right)}{\pi^{\f{s+n}{2}}} \lan{u_s, \wh{f}(\cdot) e^{2\pi i \lan{\cdot, x}}}.
\]
Furthermore, note that the constant $\Ga \left(\f{s+n}{2}\right) / \pi^{\f{s+n}{2}} \neq 0$ when $s > 0$.  Thus it suffices prove the estimates \eqref{eq:est1} and \eqref{eq:est2} for $\lan{u_s,  \wh{f}(\cdot) e^{2\pi i \lan{\cdot, x}}}$.  We extend the map $s\mapsto \lan{u_s,  \wh{f}(\cdot) e^{2\pi i \lan{\cdot, x}}}$ to an entire function and replace $s\in \br_+$ by $z\in \ccc$.

Applying Theorem~2.4.6 of \cite{G1}, we obtain
\[
\lan{u_z,  \wh{f}(\cdot) e^{2\pi i \lan{\cdot, x}}} = \lan{u_{-n-z}, f(\cdot + x)}.
\]
For $z: -n< \textnormal{Re} \,z <0$, we can use \eqref{def:uz} to write
\[
\lan{u_{-n-z}, f(\cdot + x)} =   \int_{\rn} \f{\pi^{-\f{z}{2}}}{\Ga \left(-\f{z}{2}\right)} |y|^{-n-z} f(x+y)\, dy.
\]
We split integral in the right hand side into $\int_{|y|\leq 1}\cdot \, dy + \int_{|y|> 1} \cdot\, dy =: I_1 (x,z) + I_2 (x,z)$. 

First, we recall the expression (2.4.7) in \cite{G1} which shows that $I_1(x,z)$ can be extended to an entire function in $z \in \ccc$ so that for any $z$ with $ \textnormal{Re} \, z < N$,  for some $N\in \bn$,  $I_1 (x,z)$ can be computed via the following formula:
\begin{align*}
I_1 (x,z) &= \sum_{|\al| \leq N}  b(n,\al,z) \p^{\al} \lan{\p^{\al} \de_0, f(\, \cdot \, + x)} \\
	&+ \int_{|y|<1} \f{\pi^{-\f{z}{2}}}{\Ga \left(-\f{z}{2}\right)} \left\{ f(x+y) - \sum_{|\al|\leq N} \f{(\p^{\al} f)(x)}{\al !} y^{\al} \right\}\, |y|^{-n-z}\, dy\, ,
\end{align*}
where $\al\in \bz^n_+$ is a multi-index and $b(n,\al,z)$ is an entire function for any given $n,\al$.  From this formula, we remark that for a fixed $z: \, 0<\textnormal{Re}\, z <N$, there exists $C(z, n, N)$ such that
\[
|I_1 (x,z)| \leq C(z,n,N) \left( \sum_{|\al|\leq N} {  |\p^{\al}f |(x)} +  \sup_{|y|\leq 1}  \sum_{|\be|= N+1} \sup {  |  \p^{\be} f  | (x+y)}\right).
\]
Note that $I_1 (x,z)$ decays  like a Schwartz function for any fixed $z\in \br_+$.

Now we consider $I_2(x,z)$, which is also an entire function in $z$.  For $z$ satisfying 
$ -n< \textnormal{Re}\, z<0$, this entire function is given by
\begin{equation}\label{eq:homI2}
I_2(x,z) =  \int_{|x-y|>1} \f{\pi^{-\f{z}{2}}}{\Ga \left(-\f{z}{2}\right)} |x-y|^{-n-z} f(y)\, dy.
\end{equation} 
Note that \eqref{eq:homI2} is valid for any $z\in \ccc$, so this gives an exact expression for $I_2(x,z)$.  It is important to notice that the constant $C_z := \pi^{-\f{z}{2}} /\Ga \left(-\f{z}{2}\right)$ vanishes when $z$ is a positive even integer because of   the poles of $\Ga(\cdot)$.  However, if $z \notin 2\bz_{+}$, then $C_z \neq 0$.

Let $z \in \br_+ \setminus 2\bn$. It is easily seen from \eqref{eq:homI2} that, given $z\in \br_+$, $I_2(x,z)$ is bounded for all $x\in \rn$.  Now we consider the decay rate of $I_2(x,z)$ for $|x|>2$.

Split the integral in \eqref{eq:homI2} into two regions: $I_2^1:= \int_{|x|\leq 2|y|}\cdot \, dy$ and $I_2^2:= \int_{|x| > 2|y|}\cdot \, dy$.  For the first integral, there is some constant $C(z,K)>0$ satisfying
\[
|I_2^1(x,z)| \leq |C_z| \int_{|y| \geq |x|/2} |f(y)|\, dy  = \f{|C_z|}{(1+|x|/2)^K} \int_{\rn} (1+ |y|)^K |f(y)|\, dy \leq \f{C(z,K) }{(1+ |x|)^K}
\]
for any $K\in \bn$.  Thus, over this region, $I_2(\cdot, z)$ decays like a Schwartz function.  For the remaining integral, we can drop the condition $|x-y|>1$ and write
\[
I_2^2 (x,z) = C_z \int_{|x|>2|y|}  |x-y|^{-n-z} f(y)\, dy
\]
Note that $\f{1}{2}|x|< |x-y| < \f{3}{2}|x|$ whenever $|x| > 2|y|$.  Thus, $I_2^2 (x,z) \leq C_{n,z} 1/|x|^{n+z}$ for $|x|>2$.  This proves \eqref{eq:est1}.

Moreover, if $f(y)\geq 0$ for $\forall y \in \rn$ but $f\not\equiv 0$, we have
\[
|I_2^2(x,z)| \geq \left(\f{3}{2}\right)^{n+z} \f{|C_{z}|}{|x|^{n+z}} \int_{|x|>2|y|} f(y)\, dy\, . 
\]
Taking $|x|$ large enough so that $f\not\equiv 0$ on the ball $B_{|x|/2}:= \{y\in \rn: 2|y|<|x|\}$, the integral above is bounded from below.  This proves \eqref{eq:est2}.
\end{proof}

\begin{proof}[Proof of Theorem~\ref{th:KP2} for the homogeneous case]
Note that \eqref{eq:est2} states that for all $s$ satisfying $s\neq 2k$ for some $k\in \bn$ and $0<s\leq \f{n}{r}-n$, $D^s f \not\in L^r(\rn)$.  In particular, we can choose any non-zero function $f\in \cs(\rn)$ so that $D^s |f|^2 \notin L^r(\rn)$.  On the other hand, \eqref{eq:est1} tells us that $D^s f \in L^p(\rn)$ for any $s>0$ and $p\geq 1$.  This 
disproves inequality \eqref{eq:kp} when $0<s\leq \f{n}{r}-n$, hence proves Theorem~\ref{th:KP2} in this case.  We remark also that if $s>\f{n}{r}-n$, then $D^s f\in L^r(\rn)$ for any $f\in \cs(\rn)$.

Now consider the case $\f{n}{r}-n<s<0$.  Let $\Phi, \Psi\in \cs(\rn)$ be real-valued radial functions where $\Phi$ is supported on a ball of radius $1$, and $\Psi \equiv 1$ on $\{\xi\in \rn: \,\f{1}{2}\leq |\xi| \leq 2\}$ and is supported on a larger annulus.  In \eqref{eq:kp}, let $f(x)= f_k(x):= e^{i2^k \lan{e_1, x}}\wh{\Phi}(x)$ and $g(x) =g_k(x):= e^{-i2^{k}\lan{e_1,x}} \wh{\Phi}(x)$ with $|k|\gg 1$.  Then the left hand side of \eqref{eq:kp} is independent of $k$.  On the other hand, consider the first term~$\n{D^s f}{L^p} \n{g}{L^q}$ from the right hand side.  $\n{g}{L^q}$ is independent of $k$, while 
\begin{align*}
[D^s f](x) &=  \int_{\rn} |\xi|^s \Psi(2^{-k} \xi) \Phi (\xi- 2^k e_1) e^{2\pi i \lan{ \xi, x}} \, d\xi\\
	&= 2^{ks} \int_{\rn} \Psi_s (2^{-k}\xi) \Phi (\xi - 2^k e_1) e^{2\pi i \lan{\xi, x}}\, d\xi
\end{align*}
where $\Psi_s(\cdot) := |\cdot|^{s} \Psi(\cdot)$.  Thus we have
\[
\n{D^s f}{L^p} = 2^{ks} \n{ 2^{kn} [\wh{\Psi_s} (2^{k}\cdot)] *[e^{i\lan{2^ke_1, \cdot}} \wh{\Phi}]}{L^p} \leq 2^{ks} \n{\wh{\Psi_s}}{L^1} \n{\wh{\Phi}}{L^p}.
\]
Note that since $\Psi$ is supported on an annulus, $\Psi_s \in \cs(\rn)$ so that $\wh{\Psi_s} \in \cs(\rn) \subset L^1(\rn)$.  Thus this term converges to zero as $k\to \infty$.  The second term on the right hand side of \eqref{eq:kp} is estimated similarly, which leads to a contradiction.  
\end{proof}

\begin{proof}[Proof of Theorem~\ref{th:KP1} in the homogeneous case]

Define $\Phi \in \cs(\br)$ so that $\Phi \equiv 1$ on $[-1,1]$ and is supported in $[-2,2]$.  Also let $\Psi(\xi) := \Phi(\xi) - \Phi(2\xi)$ and note that $\vp$ is supported on an annulus $\xi: \, 1/2< |\xi|<2$ and $\sum_{k\in \bz} \vp( 2^{-k} \xi)=1$ for $\forall \xi \neq 0$.

Given $f,g\in \cs(\rn)$, we decompose $D^s [fg]$ as follows:
\begin{align*}
D^s [fg] (x) &= \int_{\br^{2n}} |\xi+\eta|^s \wh{f}(\xi) \wh{g}(\eta) e^{2\pi i \lan{\xi+\eta,x}}\, d\xi \, d\eta\\
	&= \int_{\br^{2n}} |\xi+\eta|^s \left( \sum_{j\in \bz} \vp (2^{-j}\xi) \wh{f}(\xi) \right) \left( \sum_{k\in \bz} \vp (2^{-k} \eta) \wh{g}(\eta)\right) e^{2\pi i \lan{\xi+\eta,x}}\, d\xi\, d\eta\\
	&= \sum_{j\in \bz} \sum_{k: k <j-1}\int_{\br^{2n}} |\xi+\eta|^s  \vp (2^{-j}\xi) \wh{f}(\xi) \vp (2^{-k} \eta) \wh{g}(\eta) e^{2\pi i \lan{\xi+\eta,x}}\, d\xi\, d\eta\\
	&\q+\sum_{k\in \bz} \sum_{j: j <k-1}\int_{\br^{2n}} |\xi+\eta|^s  \vp (2^{-j}\xi) \wh{f}(\xi) \vp (2^{-k} \eta) \wh{g}(\eta) e^{2\pi i \lan{\xi+\eta,x}}\, d\xi\, d\eta\\
	&\q+ \sum_{k\in\bz}\sum_{j: |j-k|\leq 1} \int_{\br^{2n}} |\xi+\eta|^s  \vp (2^{-j}\xi) \wh{f}(\xi) \vp (2^{-k} \eta) \wh{g}(\eta) e^{2\pi i \lan{\xi+\eta,x}}\, d\xi\, d\eta\\
	&=: \Pi_1[f,g](x) + \Pi_2[f,g](x) + \Pi_3[f,g](x).
\end{align*}
The arguments for $\Pi_1$ and $\Pi_2$ are identical under the apparent symmetry, so it suffices to consider $\Pi_1$ and $\Pi_3$.  For $\Pi_1$, we can write
\[
\Pi_1[f,g](x)= \int_{\br^{2n}} \left\{\sum_{j\in \bz} \vp (2^{-j}\xi)\Phi (2^{-j+2} \eta) \f{|\xi+\eta|^s}{|\xi|^s}\right\}  \wh{D^s f}(\xi) \wh{g}(\eta) e^{2\pi i \lan{\xi + \eta, x}}\, d\xi\, d\eta.
\]
Since the expression in the bracket above is a bilinear Coifman-Meyer multiplier, the $\Pi_1[f,g]$ satisfies the inequality~\eqref{eq:kp}.

For $\Pi_3[f,g]$, note that the summation in $j$ is finite, thus it suffices to show estimate~\eqref{eq:kp} for the term
\begin{equation}\label{eq:pi3}
\n{\sum_{k\in\bz} \int_{\br^{2n}} |\xi+\eta|^s  \vp (2^{-k}\xi) \wh{f}(\xi) \vp (2^{-k} \eta) \wh{g}(\eta) e^{2\pi i \lan{\xi + \eta, x}}\, d\xi\, d\eta}{L^r(\rn)}.
\end{equation}
When $s\in 2\bn$, \eqref{eq:pi3} can be written as 
\[
\n{\int_{\br^{2n}} \left\{ \sum_{k\in \bz} \frac{|\xi+\eta|^s}{|\eta|^s} \vp(2^{-k}\xi) \vp(2^{-k} \eta) \right\} \wh{f}(\xi) \wh{D^s g}(\eta) e^{2\pi i \lan{\xi+\eta,x}}\, d\xi \, d\eta }{L^r(\rn)}.
\]
The expression in the bracket above belongs to Coifman-Meyer class, so the theorem follows directly in this case.   When $s\not\in 2\bn$, the symbol at hand is rougher than what is permitted from the current multilinear Fourier multiplier theorem.  At the remark at the end of this section we explain why  this symbol cannot be treated by known multiplier theorems. 

We proceed with the estimate for $\Pi_3$, which requires a more careful analysis. We have the following cases.  

\textbf{Case 1:} $\f{1}{2} < r <\infty$, $1<p,q<\infty$ or $\f{1}{2}\leq r < 1$, $1\leq p,q <\infty$.

These represent two separate cases: the former has the strong $L^r$~norm on the left hand side of \eqref{eq:kp}, and the latter has the weak $L^r$~norm instead.  However, in view of Theorem~\ref{th:CM} and Corollary~\ref{cor:m}, the strategy for the proof will be identical.  Thus we will only prove the estimate with a strong $L^r$~norm on the left hand side. Notice that when $|\xi|, |\eta|\le 2\cdot 2^k$, then $|\xi+\eta|\le  2^{k+2}$ and thus 
$\Phi (2^{-k-2}(\xi+\eta))=1$. In view of this have 
\begin{align*}
& \Pi_3 [f,g] (x)  \\
	&= \iint_{\br^{2n}} \sum_{k\in \bz} |\xi+ \eta|^s \vp(2^{-k} \xi) \wh{f} (\xi) \vp(2^{-k} \eta) \wh{g}(\eta) 
e^{2\pi i\lan{ \xi+\eta, x}}\, d\xi\,d\eta\\
	&= \iint_{\br^{2n}} \sum_{k\in \bz} |\xi+ \eta|^s \Phi (2^{-k-2}(\xi+\eta)) \vp(2^{-k} \xi) \wh{f} (\xi) \vp(2^{-k} \eta) \wh{g}(\eta) e^{2\pi i\lan{ \xi+\eta, x}}\, d\xi\,d\eta\\
	& = 2^{2s} \sum_{k\in \bz} \iint_{\br^{2n}}  \Phi_s (2^{-k-2}(\xi+\eta)) \vp(2^{-k} \xi) \wh{f} (\xi) \wtv(2^{-k} \eta)  \wh{D^s g}(\eta) e^{2\pi i\lan{ \xi+\eta, x}}\, d\xi\,d\eta \\
	&= 2^{2s}\sum_{k\in \bz} 2^{2nk} \iint_{\br^{2n}} \Phi_s (2^{-2}(\xi+\eta) ) \vp(\xi)  \wh{f}(2^k\xi) \wt{\vp}(\eta) \wh{D^s g}(2^k \eta) e^{2\pi i2^k\lan{ \xi+\eta, x}} \, d\xi\, d\eta , 
\end{align*}
where $\wtv(\cdot) := |\cdot|^{-s} \vp(\cdot)$, $\Phi_s (\cdot) := |\cdot|^s \Phi(\cdot)$.   Now the function  $\Phi_s (2^{-2}\cdot )$ is supported in $[-8,8]^n$ and can be expressed in terms of its Fourier series multiplied by the characteristic function of the set $[-8,8]^n$, denoted $\chi_{[-8,8]^n}$. 
\[
\Phi_s (2^{-2}(\xi+\eta) ) = \sum_{m\in \bz^n} c^s_{m} e^{\frac{2\pi i}{16} \lan{ \xi+\eta, m}} \chi_{[-8,8]^n} (\xi+\eta),
\]
 where $c^s_m := \f{1}{16^n}\int_{[-8,8]^n} |y|^s \Phi(2^{-2}y) e^{-\f{2\pi i}{16} \lan{y, m}}\, dy$.  Due to the support of $\vp$ and $\wtv$, we also have
\[
\chi_{[-8,8]^n} (\xi+\eta) \vp(\xi) \wtv(\eta) = \vp(\xi) \wtv(\eta),
\]
so that the characteristic function may be omitted from the integrand.  Using this identity, we   write $\Pi_3 [f,g] (x)$ as 
\begin{align*}
&= 2^{2s}\sum_{k\in \bz} 2^{2nk} \iint_{\br^{2n}} \sum_{m\in \bz^n} c^s_{m} e^{\f{2\pi i}{16} \lan{ \xi + \eta, m}} \vp(\xi) \wh{f}(2^k\xi) \wt{\vp}(\eta) \wh{D^s g}(2^k \eta) e^{2\pi i 2^k\lan{ \xi+\eta, x}} \, d\xi\, d\eta\\
&= 2^{2s}\sum_{m\in \bz^n} c^s_{m} \sum_{k\in \bz} [\De^m_k f](x)  [\wt{\De_k^{m}}  D^s g](x),
\end{align*}
where  $\De^m_k $ is the Littlewood-Paley operator given by multiplication on the Fourier transform side by $e^{2\pi i \langle{ 2^{-k} \cdot, \f{m}{16} }\rangle } \vp(2^{-k}\cdot)$, while $\wt{\De_k^{m}}$ 
is the Littlewood-Paley operator given by multiplication on the Fourier   side by 
$e^{2\pi i  \langle{ 2^{-k} \cdot, \f{m}{16}}\rangle }\wtv (2^{-k}\cdot)$.  Both 
Littlewood-Paley operators have the form: 
\[
 \int_{\rn} 2^{nk}  \Theta (2^k (x-y) +  \tfrac1{16} m) f(y)\, dy 
\]
for some Schwartz function $\Theta$ whose Fourier transform is supported in some annulus centered at zero. 

Let $r_* := \min (r,1)$.  Taking the $L^r$ norm of the right hand side above, we obtain 
\begin{align*}
&\n{D^s [f g]}{L^r}^{r_*} \\
&\leq \sum_{m\in \bz^n} |c^s_m|^{r_*} \n{  \sum_{k\in \bz} [\De_k^m f](x)  [\wt{\De^m_k} D^s g](x)}{L^r(\rn)}^{r_*}\\
&\leq \sum_{m\in \bz^n} |c^s_m|^{r_*}  \n{\sqrt{\sum_{k \in \bz} |\De_k^m f|^2}}{L^p(\rn)}^{r_*} \n{\sqrt{\sum_{k\in \bz} |\wt{\De_k^m} D^s g|^s}}{L^q(\rn)}^{r_*}
\end{align*}
whenever $\f{1}{p} + \f{1}{q} = \f{1}{r}$.  By Corollary~\ref{cor:m}, the preceding expression 
 is bounded by  a constant multiple of 
\[
 \sum_{m\in \bz^n} 
|c^s_{m}|^{r_*} [ \ln (2+ |m|) ]^{2r_*}\n{f}{L^p}^{r_*} \n{D^s g}{L^q}^{r_*}
\]
if $1<p,q<\infty$ and this term yields a constant, provided we can show that the series above converges. 

Now, applying Lemma~\ref{le:homKP},
\begin{align*}
{
c^s_m  = \int_{[-8,8]^n} |\xi|^s\Phi(2^{-2}\xi)e^{-\f{2\pi i}{16} \lan{m, \xi}}\, d\xi   = c\, [D^s \wh{\Phi(2^{-2}\cdot)}] ( \tfrac{m}{16}) = O((1+|m|)^{-n-s})  }
\end{align*} 
as $|m|\to \infty$ and $c^s_m$ is uniformly bounded for all $m\in \bz$.  Thus, since $r_*(n+s)>n$, the series  $\sum_{m\in \bz^n} |c_m^s|^{r_*}  [\ln (1+ |m|)]^{2r_*}$ converges.  This concludes Case 1.

\textbf{Case 2:} $1<r<\infty$, $(p, q) \in \{(r,\infty), (\infty,r)\}$

 Here we provide a proof, which is an adaptation of the proof given in \cite[Section 6.2]{BB}.  This method will extend more readily to the multi-parameter case, which is presented in Section~\ref{sec:multiKP}.  Write
\[
\n{\Pi_3[f,g]}{L^r(\rn)} \leq C(r,n) \n{\left(\sum_{j\in\bz} |\De_j \Pi_3[f,g]|^2\right)^{\f{1}{2}} }{L^r(\rn)}.
\]
The summand in $j$ above can be estimated as follows
\begin{align*}
&\De_j \Pi_3[f,g] (x)\\
&= \int_{\br^{2n}} |\xi+\eta|^s \Psi (2^{-j} (\xi+\eta)) \sum_{k\geq j-2}\Psi(2^{-k} \xi) \wh{f}(\xi) \Psi(2^{-k} \eta) \wh{g}(\eta)\, e^{2\pi i \lan{  \xi+\eta, x}}\, d\xi\, d\eta\\
	&= \int_{\br^{2n}}2^{js} \wt{\Psi_s} (2^{-j} (\xi+\eta))  \Psi(2^{-k} \xi) \wh{f}(\xi) \sum_{k\geq j-2} 2^{-ks} \wt{\Psi_{-s}}(2^{-k} \eta) \wh{D^s g}(\eta)\, e^{2\pi i \lan{  \xi+\eta, x}}\, d\xi\, d\eta\\
	&= 2^{js} \sum_{k\geq j-2}2^{-ks} \wt{\De_j^s} \left[   [\De_k f] [\wt{\De_k^{-s}} D^s g]\right](x)\\
	&\leq 2^{js} \left(\sum_{k\geq j-2} 2^{-2ks}\right)^{\f{1}{2}} \left( \sum_{k\geq j-2} \abs{\wt{\De_j^{s}} \left[ [\De_k f] [\wt{\De_k^{-s}} D^s g]\right](x)}^2\right)^{\f{1}{2}}\\
	&\leq 
	C(s)\left( \sum_{k\geq j-2} \abs{\wt{\De_j^{s}} \left[ [\De_k f] [\wt{\De_k^{-s}} D^s g]\right](x)}^2\right)^{\f{1}{2}}\, , 
\end{align*}
where $\wt{\Psi_{s}}(\cdot) := |\cdot|^s \Psi(\cdot)$ and $\cf [\wt{\De_k^{s}} f](\cdot) := \wt{\Psi_s}(2^{-k}\cdot) \wh{f}(\cdot)$. Thus we   have
\begin{align*}
\n{\Pi_3[f,g]}{L^r} &\leq C(r,n,s)
\n{ \left(\sum_{j\in \bz} \sum_{k\in \bz} \left|\wt{\De_j^s}[ \De_k f  \, \wt{\De_k^{-s}} D^s g] \right|^2\right)^{\f{1}{2}}}{L^r}.
\end{align*}
We apply \cite[Proposition 4.6.4]{G1} to extend $\{\wt{\De_k^{s}}\}_{k\in \bz}$ from $L^r \to L^r \ell^2$  to $L^r \ell^2 \to L^r \ell^2 \ell^2$ for $1<r<\infty$.  This gives
\begin{align*}
\n{\Pi_3[f,g]}{L^r}	&\leq C(r,n,s) \n{\left(\sum_{k\in \bz} \left|\De_k f  \, \wt{\De_k} D^s g \right|^2\right)^{\f{1}{2}}}{L^r}\\
	&\leq C(r,n,s)\n{\sup_{k\in \bz} \wt{\De_k} D^s g}{L^{\infty}} \n{\left(\sum_{k\in \bz} \left|\De_k f \right|^2\right)^{\f{1}{2}}}{L^r}\\
&\leq C(r,n,s) \sup_{k\in \bz} \n{\wt{\De_k} D^s g}{L^{\infty}} \n{f}{L^r} \\
&\leq C(r,n,s)   \Big\| \wh{\wt{\Psi_{-s}}}\Big\|_{L^1} \n{D^s g}{L^{\infty}}\n{f}{L^r}.
\end{align*}
This proves the case when $(p,q) = (r,\infty)$and the case $(p,q) = (\infty,r)$ follows by symmetry.
\end{proof}

We emphasize that the bound for $\Pi_3$ does not follow from presently-known multiplier estimates. In the following remark, we consider the symbol of $\Pi_3$ under the smoothness criteria given in \cite{BNT, CM, GS, Tomita}.  Define  
\[
\si_s(\xi,\eta) := \sum_{k\in \bz} \f{|\xi+\eta|^s}{|\xi|^s} \vp(2^{-k}\xi) \vp(2^{-k} \eta).
\]
Then $\Pi_3[f,g] = T_{\si_s} [f, D^s g]$, where $T_{\si_s}$ is the pseudo-differential operator with symbol $\si_s$.  We observe the following.

\begin{remark}
For $s\in \br_+ \setminus 2\bn$, 
\begin{enumerate}
\item $\si_s$ does not belong to the Coifman-Meyer class (Theorem~\ref{th:CM}) for $s<2n$.

\item $\si_s \in \dot{BS}^{0}_{1,0; -\pi/4}$ defined below,  which forms a degenerate class of pseudo-differential operators, \cite{BNT}.

\item  $\si_s$ fail the smoothness conditions in \cite{GS} if $s < n/2$.
\end{enumerate}
\end{remark}

\begin{proof}[Sketch of proof of Remark 1]
Firstly, it is easily seen that $\p^{\al} \si_s$ develops a singularity on the hyperplane $\xi+\eta = 0$ whenever $|\al| > s$, so the bound
$|\p_{\xi_1}^{2n+1} \si_s|(\xi,\eta) \leq C (|\xi|+ |\eta|)^{-2n-1}$ from Theorem~\ref{th:CM} is not satisfied when $s<2n$ for any $\xi,\eta \in \rn\setminus \{0\}$ with $\xi  +\eta =0$. 

Secondly, we recall the class of bilinear symbols denoted $BS^0_{0,1; \theta}$ for $\theta \in (-\pi/2,\pi/2]$ given in \cite{BNT}: $\si(\xi,\eta)\in BS^0_{1,0; \theta}$ if
\[
|\p_{\xi}^{\al} \p_{\eta}^{\be} \si| (\xi,\eta) \leq C_{\al,\be} (1+ |\eta - \xi \tan \theta|)^{-|\al|-|\be|};
\]
and $\si \in \dot{BS}^{0}_{1,0;\theta}$ if RHS above can be replaced by $C_{\al,\be} |\eta - \xi \tan \theta|^{-|\al|-|\be|}$.  In \cite{BNT}, B\'enyi, Nahmod, Torres remarked that  the classes of symbols $BS^0_{1,0; \theta}$ and $\dot{BS}^0_{1,0;\theta}$ given by the angle $\theta = -\pi/4$ (along with $\theta = 0$ and $\theta = \pi/2$) are degenerate in the sense that the boundedness results for singular multipliers do not apply anymore, even for $r>1$.  As we noted above, $\si_s$ develops a singularity along $\xi+ \eta = 0$, which is permitted only in the class $\dot{BS}^{0}_{1,0; -\pi/4}$.  So we can easily note that $\si_s \in \dot{BS}^{0}_{0,1;\theta} \iff \theta = -\pi/4$.  This portrays that $\si_s$ does not belong to a known bounded class of singular multiplier symbols.

Lastly, we show that $\si_s$ fails to satisfy the condition given in \cite{Tomita} and also in  \cite{GS}   when $s < n$.  Let $\Psi \in \cs(\br^{2n})$ be defined such that $\Psi$ is supported on an annulus $1/2\le |(\xi,\eta)| \le 2$ and $\sum_{j\in \bz} \Psi(2^{-j} \xi, 2^{-j} \eta) = 1$ except at the origin.  The condition in \cite{Tomita} and \cite{GS} requires that for some $\ga>n$,
\[
\sup_{j\in \bz} \n{\Psi(\cdot) \si_s (2^{j} \,\cdot)}{L^2_{\ga}(\br^{2n})} < \infty.
\]
We will show that right hand side of above is infinite when $s\in \br_+ \setminus 2\bz_{+}$, $s<n/2$ and $\ga = n$.
  
Note that $\si_s( 2^j \xi, 2^j \eta) =\si_s (\xi,\eta)$ for any $j\in \bz$ and also that multiplying $\Psi$ to $\si_s$ reduces the summation in $k$ to a finite sum $\sum_{|k|\leq 1}$.  Thus, 
\[
\Psi(\xi,\eta) \si_s(2^j \xi,2^j \eta) = |\xi+\eta|^s \wt{\Psi}(\xi,\eta)
\]
where 
\[
\wh{\Psi}(\xi,\eta) := \Psi(\xi,\eta) \sum_{|k|\le 1} |\xi|^{-s} \vp(2^{-k}\xi) \vp(2^{-k} \eta).  
\]
Note that for any $s\in \br\setminus \{0\}$,
\[
\De_{\xi} |\xi + \eta|^s = s(s+n-2) |\xi+\eta|^{s-2}\; \qq \left|\na_{\xi} |\xi+\eta|^s\right| = s |\xi+\eta|^{s-1}.
\]
For the derivatives on $|\xi+\eta|^s\wt{\Psi}(\xi,\eta)$, the singularity along the hyperplane $\{\xi +\eta =0\}$ is dominated by $\p^{\al} |\xi+\eta|^s$ when $|\al|>s$.  For our purpose, we can assume (near this hyperplane) that all derivatives fall on the rough term $|\xi+\eta|^s$.

First, let $n = 2k$ for some $k\in \bn$.  To see that $|\xi+\eta|^s \wt{\Psi}(\xi,\eta) \notin L^2_n(\br^{2n})$, there is some $C(s,n)\neq 0$ such that
\[
[\De_{\xi}^{k} |\xi+\eta|^s] \wt{\Psi}(\xi,\eta) = C(s,n) |\xi+\eta|^{s-n} \wt{\Psi}(\xi,\eta).
\]
Also note that $\wt{\Psi}$ is bounded below by some constant within the set 
\[
\left\{ (\xi,\eta)\in \br^{2n}: \, |\xi+\eta| \leq \f{1}{10},  \f{3}{4} \le |\eta|\le \f{3}{2} \right\}
\]
Thus, 
\begin{align*}
\n{\Phi \si_s}{L^2_n} &\geq C(n) \int_{\br^{2n}} \left|\De_{\xi}^{k} \left[|\xi+\eta|^s\right] \wt{\Psi}(\xi,\eta) \right|^2\, d\xi\, d\eta\\
 &\geq C(s,n) \int_{\eta: \f{3}{4} \leq |\eta|\leq \f{3}{2}} \int_{\xi: |\xi+\eta|\leq \f{1}{10}} |\xi+\eta|^{2s-2n}\, d\xi\, d\eta\\
	&= C(s,n) \int_{|\xi|\leq \f{1}{10}} |\xi|^{2s-2n}\, d\xi
\end{align*}
which is infinite when $2s-2n < -n$ (i.e. $s<n/2$).  

If $n= 2k+1$ for some $k\in \bn$, then we can make the same argument as above after replacing $\De_{\xi}^{k}$ by $\left| \na_{\xi} [ \De_{\xi}^k |\xi+\eta|^s]\right|$.  
\end{proof}

\section{Inhomogeneous Kato-Ponce inequality}\label{sec:inhKP}
In this section, we   discuss the original Kato-Ponce inequality  \eqref{eq:KatoPonce}.  Our approach is largely based on the idea developed in the previous section for the homogeneous symbol and does not depend on the smoothness of the symbol $(1+|\cdot|^2)^{s/2}$.  It is indeed surprising that not only the positive results, but also the negative results parallel the ones given for the homogeneous symbol.  This phenomenon is sharply observed in the following lemma.

\begin{lemma}\label{le:inhom}
Let $f\in \cs(\rn)$ and $s>0$.  Then for all $\de \in (0,1]$, there exists a constant $C(n,s,f)$ independent of $\de$ such that
\begin{equation}\label{eq:Bessel}
|(\de^2-\De)^{s/2} f| (x) \leq C(n,s,f) (1+|x|)^{-n-s}
\end{equation}
\end{lemma}
\textbf{Remark:}  In \cite{CZ}, the authors use scaling of \eqref{eq:KatoPonce} and dominated convergence to deduce \eqref{eq:kp} from \eqref{eq:KatoPonce}.  This lemma gives a justification for this argument as well.  In fact, these two inequalities are intricately related due to the heuristic relationship  
$$
J^s := (1-\De)^{s/2} \approx 1 + (-\De)^{s/2} =: 1+D^s.
$$
  This approximation is rigorous when working in $L^p$ for $1 \le p<\infty$, so that \eqref{eq:kp} may imply \eqref{eq:KatoPonce} as well in this case.

\begin{proof}
First, note that $[(\de^2 -\De)^{s/2} f]$ is uniformly bounded in $\de \in (0,1]$ since
\[
\left|\int_{\rn} (\de^2 + |\xi|^2)^{\f{s}{2}} \wh{f}(\xi)e^{2\pi i\lan{\xi, x}}\, d\xi\right| \leq \int_{\rn} (1+|\xi|)^{\f{s}{2}} |\wh{f}|(\xi)\, d\xi <\infty.
\]

It remains to show that for $|x|\geq 1$, $|(\de^2 -\De)^{s/2} f| (x)\leq C(n,s,f) |x|^{-n-s}$.

For $z \in \ccc$ and $\de \in (0,1]$, define the distribution $v_z^{\de}$ by the action
\[
\lan{v_z^{\de}, f} = \int_{\rn} (\de^2 + |\xi|^2)^{\f{z}{2}} f(\xi) \, d\xi
\]
for $f\in \cs(\rn)$.  Note that the map $z\mapsto \lan{v_z^{\de}, f}$ defines an entire function.  If $z\in \br$ and $z <0$ and $\de =1$, $\wh{v_z^{\de}}$ is known as the Bessel potential, denoted $G_{-z}$, given in \cite[Chapter 6]{G1}.  We now extend the distribution $\wh{v^{\de}_z}$ to $z\in \ccc$.

Begin with the Gamma function identity: for $A>0$ and $z:\,\RE z<0$
\[
A^z = \f{1}{\Ga (-z)} \int_0^{\infty} e^{tA} t^{-z -1}\, dt.
\]

Consider the map $z\mapsto \lan{v_z^{\de},\wh{f}}$ when $z:\, \RE  z <0$.  Using the identity above, we have
\begin{align*}
\lan{v_z^{\de}, \wh{f}} : &= \int_{\rn} \f{1}{\Ga \left(-\f{z}{2}\right)} \int_0^{\infty} e^{-\de^2 t} e^{-t|\xi|^2} t^{-\f{z}{2} -1}\, dt \wh{f}(\xi) \,  d\xi\\
&= \f{1}{\Ga \left(-\f{z}{2}\right)} \int_{\rn} \int_0^{\infty} e^{-\f{|y|^2}{t}}  f(y) e^{-\de^2 t} t^{-\f{z+n}{2}}\, \f{dt}{t}\, dy.
\end{align*}
Denote $K_z^{\de}(y) :=  \int_0^{\infty} e^{-\f{|y|^2}{t}} e^{-\de^2 t} t^{-\f{z+n}{2}}\, \f{dt}{t}$ to be the kernel above.  First we observe,
\[
|K^{\de}_z| (y) = |y|^{-\RE z-n}\left|\int_0^{\infty} e^{-\f{1}{t}}  e^{-\de^2 |y|^2 t} t^{-\f{z+n}{2}}\, \f{dt}{t}\right| \leq C(z,n) |y|^{-\RE z-n}
\]
for all $y\in \rn\setminus \{0\}$ when $\RE\, z>-n$.  Recall from \cite[Proposition 6.1.5]{G1} that $|K^{\de}_z|(y) \sim \ln |y|^{-1}$ when  $\RE\, z = -n$; and $|K^{\de}_z|(y) \sim 1$ when $\RE\, z < -n$.  

Given $\de>0$, the kernel also satisfies a better asymptotic estimate for sufficiently large $|y|$.  For $|y| >\de^{-1}$ and $\de \leq 1$, note that $\de^2 t + \f{|y|^2}{t} \geq \max \left(\de^{2} t + \f{1}{\de^2 t}, 2\de |y|\right)$.  Thus 
\begin{align*}
|K_z^{\de}|(y) &\leq e^{-\de |y|} \int_0^{\infty} e^{- \f{1}{2}\de^2 t +\f{1}{2\de^{2} t}} t^{-\f{\RE z+n}{2}}\, \f{dt}{t}
	&= \de^{\RE z+n}  e^{-\de |y|} \int_0^{\infty} e^{-\f{1}{2}(t+ \f{1}{t})} t^{-\f{\RE z +n}{2}}\, \f{dt}{t}.
\end{align*}
Since the integral above converges for any $z\in \ccc$, we have that for all $|y|> \de^{-1}$, $|K_z(y)| \leq C(z,n) \de^{\RE z +n} e^{-\de |y|}$.  We remark that $\sup_{\de \in (0,1]} \de^{\RE z +n} e^{-\de |y|} = c\,|y|^{-n-\RE z}$.

Note that $K_z(y)$ is not locally integrable when $\RE z \geq 0$ so that it is not well-defined as a tempered distribution.  However, since $z\mapsto \lan{v_z^{\de},\wh{f}}$ is an entire function, it suffices to find a holomorphic extension of $\lan{\wh{v_z^{\de}}, f}$ which is defined as $\lan{K_z^{\de}, f}$ for $\RE z <0$.  We continue
\begin{align}\label{eq:I1}
\lan{\wh{v_z^{\de}},f} &=  \f{1}{\Ga \left(-\f{z}{2}\right)} \int_{\rn}K^{\de}_z(y)\, \left[f(y) - \sum_{|\al| < N} \f{[\p^{\al} f](0)}{\al !} y^{\al}\right]\, dy\\
	&\q+ \sum_{|\al| < N}\f{ [\p^{\al} f](0)}{\al !}\f{1}{\Ga \left(-\f{z}{2}\right)} \int_{\rn}   y^{\al}  \int_0^{\infty}e^{- \f{|y|^2}{t}} e^{-\de^2 t} t^{-\f{z+n}{2}}\, \f{dt}{t}\, dy \notag\\
	&=: I_1^{\de} (z) + I_2^{\de} (z).\notag
\end{align}

Consider first $I_2^{\de}$.
\begin{align}\label{eq:I2}
I_2^{\de} (z) &= \sum_{|\al| < N} \f{[\p^{\al} f](0)}{\al !}\f{1}{\Ga \left(-\f{z}{2}\right)} \int_{\rn}   y^{\al}  \int_0^{\infty} e^{- \f{|y|^2}{t}}e^{-\de^2 t} t^{-\f{z+n}{2}}\, \f{dt}{t}\, dy\notag\\
	&= \sum_{|\al| < N} \f{[\p^{\al} f](0)}{\al !}\f{1}{\Ga \left(-\f{z}{2}\right)}  \int_0^{\infty} e^{-\de^2 t} t^{\f{|\al| - z}{2}}\, \f{dt}{t}\,\int_{\rn}   y^{\al} e^{-|y|^2}\, dy\notag\\
 	&= \sum_{|\al| < N} \f{\de^{z-|\al|}}{\al !} [\p^{\al} f](0)\f{1}{\Ga \left(-\f{z}{2}\right)} \int_0^{\infty} e^{- t} t^{\f{|\al| - z}{2}}\, \f{dt}{t} \, \int_{\rn} y^{\al} e^{- |y|^2} \, dy \notag\\
	&=\sum_{|\al| < N}  \f{\de^{z-|\al|}}{\al !}[\p^{\al} f](0) \f{\Ga \left(\f{|\al| -z}{2}\right)}{\Ga \left(-\f{z}{2}\right)} \,\int_0^{\infty} r^{n + |\al|-1}  e^{-r^2} \,dr  \,\int_{S^{n-1}} \theta^{\al} \, d\theta.
\end{align}
Note that the integral $\int_{S^{n-1}} \theta^{\al} \, d\theta$ vanishes unless $\al_j$ is even for all $j=1,2,\dots, n$ where $\al = (\al_1,\al_2, \dots, \al_n)$.  In this case, $|\al|$ is even.  Therefore, the poles of the function 
$\Ga((|\al|-z)/2)$ cancel with the poles of $\Ga (-z/2)$, and  thus $I_2^{\de} (z)$ is entire.  

We should remember that when $z$ is a positive even integer, the poles of $\Ga (-\cdot/2)$ make  the term $I^{\de}_1 (z)$ vanish identically.  Thus, as expected in this case, $\wh{v_z^{\de}}$ yields a local differential operators composed of only even-order derivatives.

Next we turn our attention to $I_1^{\de}(z)$. The expression inside the square bracket on the right hand side of \eqref{eq:I1} is locally $O(|y|^N)$.  Since $K^{\de}_z(y) = O(|y|^{-n -\RE z})$, the integral converges locally (say $|y|\leq \de^{-1}$) if $\RE z <N$.  For $|y| > \de^{-1}$, the kernel $K^{\de}_z (y)$ decays exponentially, while the expression inside the square bracket grows at most like $O(|y|^{N-1})$.  Thus, the integral converges.

Therefore, we note that \eqref{eq:I1} and \eqref{eq:I2} extends the function $z\mapsto \lan{K_z, f}$ holomorphically on the half plane $z: \, \RE z <N$.  Thus, this defines the tempered distribution $\wh{v^{\de}_z}$.

Now consider $f^s_{\de} := (\de^2 - \De)^{s/2} f$ for $s>0$ and $\de>0$.  We can write $f^s_{\de}(x) = \lan{v^{\de}_s, \wh{f} e^{i\lan{\cdot,x}}} = \lan{\wh{v^{\de}_s}, f(\cdot + x)}$.  Let $s\in [N-1, N)$ for some $N\in \bn$.  Using the formula \eqref{eq:I1} and \eqref{eq:I2}, $\lan{\wh{v^{\de}_s}, f(\cdot + x)}$ can be expressed as
\[ 
C(s) \int_{\rn}K^{\de}_s(y)\, \left[f(x+y) - \sum_{|\al| < N} \f{[\p^{\al} f](x)}{\al !} y^{\al}\right]\, dy +\sum_{|\al| < N} C(\al, n) \de^{s-|\al|}[\p^{\al} f](x)\, , 
\]
where the first constant $C(s) = 0$ when $s$ is a positive even integer.

The second term above is a Schwartz function, and decays uniformly in $\de\in (0,1]$ since $s-|\al|\geq 0$ when $|\al|\leq N-1$.

For the first term, we split the integral into two parts $\int_{|y|< 1}\cdot\, dy + \int_{|y|\geq 1}\cdot \, dy =: J_1(x) + J_2(x)$. We have that 
\begin{align*}
J_1(x) &= \int_{|y|<1}K^{\de}_s(y)\, \left[f(x+y) - \sum_{|\al| < N} [\p^{\al} f](x) y^{\al}\right]\, dy \\
	&\leq \sup_{|\be| = N} \sup_{|y'| <1} |\p^{\be} f|(x+ y') \,\int_{|y|< 1}|y|^{-n-s+N} \, dy.
\end{align*}
Since $-n-s+N >-n$, the last integral above is convergent.  Also, we note that the expression $\sup_{|\be| = N} \sup_{|y'| <1} |\p^{\be} f|(x+ y')$ decays like a Schwartz function.

The estimate for $J_2$ is more delicate.  We need to consider separately the case $s= N-1$ and $s\in (N-1,N)$.  First, consider when $s\in (N-1,N)$.  In this case,
\begin{align*}
J_2(x) &= \int_{ |y|\geq 1}K^{\de}_s(y)\, \left[f(x+y) - \sum_{|\al| < N} [\p^{\al} f](x) y^{\al}\right]\, dy \\	
&= \int_{|y|\geq 1}|y|^{-n-s} |f|(x+y)\, dy  + \sum_{|\al| \leq  N-1} |\p^{\al} f|(x) \int_{ |y|\geq 1}  |y|^{-n-s + |\al |}\, dy .
\end{align*}
For the first term
\begin{align*}
\int_{|y|\geq 1}|y|^{-n-s} |f|(x+y)\, dy &\leq |x|^{-M}\int_{2|y|\leq |x|} |x+y|^M |f|(x+y)\, dy\\
	&\q +C|x|^{-n-s} \int_{2|y|> |x|} |f|(x+y)\, dy .
\end{align*}
for any $M\in \bn$.  Thus this decays like $|x|^{-n-s}$.

For the second term, the integral $\int_{ |y|\geq 1}  |y|^{-n-s + |\al |}\, dy$ converges since $-n-s + |\al| < -n$ when $|\al| \leq N+1$ and $s\in (N-1, N)$, so that the second term decays like a Schwartz function. 

Now consider the special case when $s = N-1$.  Note that $s$ has to be an odd integer, since otherwise, the terms $J_1, J_2, J_3$ would not even appear due to the vanishing constant $C(s)$ mentioned above. Since $K_z^{\de}(y)$ is a radial function (and exponentially decaying), the integral
\[
\int_{|y|\geq 1} K^{\de}_z (y) y^{\al} \, dy = \int_1^{\infty} K^{\de}_z (r)\, r^{|\al|+ n-1} \, dr \, \int_{S^{n-1}} \theta^{\al}\, d\theta = 0
\]
since $|\al|$ is odd.  This concludes the proof of Lemma \ref{le:inhom}.
\end{proof}

\begin{proof}[Proof of Theorem~\ref{th:KP2} for the inhomogeneous case]
Fix an  index $\f{1}{2}<r<1$ and indices $1< p_1, p_2, q_1, q_2\leq \infty$ satisfying $\f{1}{p_1}+ \f{1}{q_1}= \f{1}{r}= \f{1}{p_2} + \f{1}{q_2}$; also fix $0<s\leq \f{n}{r}-n$.  

Assume    that \eqref{eq:KatoPonce} holds and we will reach a contradiction.  Scaling $x \mapsto \la x$ for some $\la>0$, this inequality is equivalent to
\begin{align}\begin{split}\label{eq:scaleKP}
&\n{(\la^{-2}-\De)^{s/2} [fg]}{L^r} \leq  \\
& \qq\qq C\left( \n{(\la^{-2}-\De)^{s/2} f}{L^{p_1}} \n{g}{L^{q_1}} +  \n{f}{L^{p_2}} \n{(\la^{-2}-\De)^{s/2}g}{L^{q_2}} \right).
\end{split}\end{align}

By Lemma~\ref{le:inhom}, we know that the functions $[(\la^{-2} - \De)^{s/2} f](x)$ and $[(\la^{-2}-\De)^{s/2} g](x)$ are pointwise dominated by a constant multiple of $(1+ |x|)^{-n-s}$ uniformly for $\la>1$.  Then the right hand side of \eqref{eq:scaleKP} is bounded uniformly in $\la> 1$.

On the other hand, $(\la^{-2} - \De)^{s/2} [fg] \to D^{s} [fg]$ pointwise everywhere by Lebesgue dominated convergence.  By Fatou's lemma, this implies that
\[
\int_{\rn} \left|D^{s} [fg]\right|^{r}\,dx \leq \liminf_{\la\to \infty} \int_{\rn} \left|(\la^{-2} -\De)^{s/2} [fg]\right|^r\, dx.
\]
Since $D^{s} [fg]\notin L^r(\rn)$ if $0<s\leq  n/r -n$, the left hand side of \eqref{eq:scaleKP} is infinite, which leads to a contradiction.

When $\f{n}{r}-n<s<0$, consider the counter-example given in Section~\ref{sec:homKP}.  The left hand side of \eqref{eq:KatoPonce} is independent of $k$, while the $\n{J^s f}{L^p}$ term on the right side can be written as 
\begin{align*}
[J^s f](x) &= \int_{\rn} (1+|\xi|^2)^{\f{s}{2}} \Phi (\xi - 2^k e_1) e^{2\pi i \xi \cdot x}\, d\xi \\
	&= \int_{\rn} 2^{ks} \Psi_s^k (\xi)\Phi(\xi - 2^k e_1)e^{2\pi i \xi \cdot x} \, d\xi
\end{align*}
where $\Psi_s^k(\cdot) := (2^{-2k} + |\xi|^2)^{\f{s}{2}} \Psi(\cdot)$.  Then
\[
\n{J^s f}{L^p} \leq 2^{ks} \big\|{\wh{\Psi_s^k}}\big\|_{L^1} \big\|{\wh{\Phi}}\big\|_{L^p}.
\]
The fact that $\big\| {\wh{\Psi_s^k}}\big\|_{L^1}$ is uniformly bounded is shown below in Lemma~\ref{le:decay} and the remark following.  Taking $k\to \infty$, we arrive at a contradiction.
\end{proof}

\begin{proof}[Proof of Theorem~\ref{th:KP1} for the inhomogeneous case]
We resume the notations $\Phi$, $\Psi$ introduced in Section~\ref{sec:homKP}.  Via similar computations, we split the estimate above into $\Pi_1, \Pi_2, \Pi_3$.  More precisely,
\begin{align*}
J^s[fg](x) &= \sum_{j\in \bz} \sum_{k: k <j-1}\int_{\br^{2n}} (1+|\xi+\eta|^2)^{\f{s}{2}}  \vp (2^{-j}\xi) \wh{f}(\xi) \vp (2^{-k} \eta) \wh{g}(\eta) e^{2\pi i \lan{\xi + \eta, x}}\, d\xi\, d\eta\\
	&\q+\sum_{k\in \bz} \sum_{j: j <k-1}\int_{\br^{2n}} (1+|\xi+\eta|^2)^{\f{s}{2}}  \vp (2^{-j}\xi) \wh{f}(\xi) \vp (2^{-k} \eta) \wh{g}(\eta) e^{2\pi i \lan{\xi + \eta, x}}\, d\xi\, d\eta\\
	&\q+ \sum_{k\in\bz}\sum_{j: |j-k|\leq 1} \int_{\br^{2n}} (1+|\xi+\eta|^2)^{\f{s}{2}}  \vp (2^{-j}\xi) \wh{f}(\xi) \vp (2^{-k} \eta) \wh{g}(\eta) e^{2\pi i \lan{\xi + \eta, x}}\, d\xi\, d\eta\\
	&=: \Pi_1[f,g](x) + \Pi_2[f,g](x) + \Pi_3[f,g](x).
\end{align*}

As described in Section~\ref{sec:homKP}, estimates for $\Pi_1$ and $\Pi_2$ follow from Theorem~\ref{th:CM}.  More specifically,
\[
\Pi_1[f,g](x)= \int_{\br^{2n}} \left\{\sum_{j\in \bz} \vp (2^{-j}\xi)\Phi (2^{-j+2} \eta) \f{(1+|\xi+\eta|^2)^{\f{s}{2}}}{(1+|\xi|^2)^{\f{s}{2}}}\right\} \wh{J^s f}(\xi) \wh{g}(\eta) e^{2\pi i \lan{\xi + \eta, x}}\, d\xi\, d\eta.
\]
where the symbol inside the bracket satisfies the Coifman-Meyer condition given in Theorem~\ref{th:CM}.

Thus it suffices to estimate $\Pi_3$. { For simplicity we only consider the term $j=k$.} We need to control the  following term:
\[
\sum_{k\in \bz} \int_{\rn} (1+|\xi+\eta|^2)^{\f{s}{2}} \vp(2^{-k} \xi)\wh{f}(\xi) \vp(2^{-k} \eta) \wh{g}(\eta) e^{2\pi i\lan{\xi+\eta, x}} \, d\xi\, d\eta
\]
The inhomogeneous estimate is slightly different because we cannot transfer the derivatives from the product to the high frequency term simply via scaling.  More specifically, recall that the key step in the homogeneous estimate was the identity
\[
|\xi+\eta|^s  = 2^{ks} |2^{-k} (\xi+\eta)|^s |2^{-k}\eta|^{-s} |2^{-k} \eta|^{s} = |2^{-k} (\xi+\eta)|^s |2^{-k}\eta|^{-s} |\eta|^{s}.
\]
When repeated for this setting, we obtain
\[
(1+|\xi+\eta|^2)^{\f{s}{2}} = (2^{-2k} + |2^k (\xi+\eta)|^2)^{\f{s}{2}} (2^{-2k} + |2^k \eta|^2)^{-\f{s}{2}}  (1+|\eta|^2)^{\f{s}{2}}.
\]
Using \eqref{eq:Bessel}, we can control these terms when $k\geq 0$, but the constant $2^{-2k}$ grows unboundedly when $k<0$.  Thus, we need to separate these cases.

On the other hand, when $k<0$, we note that the term $(1+|\eta|^2)^{-\f{s}{2}}$ remains bounded when $\eta \sim 2^k$.  This advantage will enable us to handle this case.

We split into the following cases as in Section~\ref{sec:homKP}.

\textbf{Case 1:} $\f{1}{2} < r <\infty$, $1<p,q<\infty$ or $\f{1}{2}\leq r < 1$, $1\leq p,q <\infty$.

As before, we will only show the estimate for the first case, whereas the estimates involving weak $L^r$ norms will immediately follow when the corresponding norms are replaced in the proof below.

First consider the sum when $k\geq 0$.
\begin{align*}
&\Pi_3^1[f,g] (x)  \\
: &= \sum_{k\geq 0} \int_{\rn} (1+|\xi+\eta|^2)^{\f{s}{2}} \vp(2^{-k} \xi)\wh{f}(\xi) \vp(2^{-k} \eta) \wh{g}(\eta) e^{2\pi i\lan{\xi+\eta, x}} \, d\xi\, d\eta\\
	&= \sum_{k\geq 0} \int_{\rn} 2^{ks} (2^{-2k}+|2^{-k}(\xi+\eta)|^2)^{\f{s}{2}}\Phi(2^{-k-2} (\xi+\eta))  \\
	&\hspace{180pt} \vp(2^{-k} \xi)\wh{f}(\xi) \vp(2^{-k} \eta) \wh{g}(\eta) e^{2\pi i\lan{\xi+\eta, x}} \, d\xi\, d\eta\\
	&= \sum_{k\geq 0} \int_{\rn} \Phi^k_s(2^{-k} (\xi+\eta)) \vp(2^{-k} \xi)\wh{f}(\xi)  \\
	&\hspace{100pt} (2^{-2k} + |2^{-k}\eta|^2)^{-\f{s}{2}}\vp(2^{-k} \eta) (1+ |\eta|^2)^{\f{s}{2}} \wh{g}(\eta) e^{2\pi i\lan{\xi+\eta, x}} \, d\xi\, d\eta , 
\end{align*}	
where $\Phi^k_s(\cdot) := (2^{-2k} + |\cdot|^2)^{s/2} \Phi(2^{-2} \cdot)$.  Since $\vp$ is supported on an annulus,   pick $\wt{\Psi} \in \cs(\rn)$   equal to one on the support of $\vp$ and  supported { in the slightly larger annulus $\f 12- \f1{10}\le |\xi|\le 2+\f 15$}.  Writing $\vp= \wtv \vp$, we obtain that 
\begin{align*}
&\Pi_3^1[f,g] (x)  \\
	&= \sum_{k\geq 0} \int_{\rn} \Phi^k_s(2^{-k} (\xi+\eta)) \vp(2^{-k} \xi)\wh{f}(\xi)   \\
	&\hspace{110pt}(2^{-2k} + |2^{-k}\eta|^2)^{-\f{s}{2}}\wtv(2^{-k}\eta) \vp(2^{-k} \eta)\wh{J^s g}(\eta) e^{2\pi i\lan{\xi+\eta, x}} \, d\xi\, d\eta\\
	&= \sum_{k\geq 0} \int_{\rn} \Phi^k_s(2^{-k} (\xi+\eta)) \vp(2^{-k} \xi)\wh{f}(\xi) \wtv^k_{-s} (2^{-k}\eta) \vp(2^{-k} \eta)  \wh{J^s g}(\eta) e^{2\pi i\lan{\xi+\eta, x}} \, d\xi\, d\eta, 
\end{align*}
where $\wtv^k_{-s}(\cdot) := (2^{-2k}+|\cdot|^2)^{s/2} \wtv(\cdot)$.  
Now we scale back and expand $\Phi^k_s$ and $\wtv^k_{-s}$ in  Fourier series, as in Section~\ref{sec:homKP}, over the cube $[-8,8]^n$. Let 
 $c_{m,k}^s$ and  $\wt{c}^{-s}_{l,k}$ be the Fourier coefficients of the expansion 
 defined as
\begin{align*}
c_{m,k}^s &:= \f{1}{8^n} \int_{[-8,8]^n} (2^{-2k} + |\xi|^2)^{s/2} \Phi(\xi) e^{-\f{2\pi i}{16}\lan{\xi, m}} \, d\xi\\
\wt{c}^{-s}_{l,k} &:= \f{1}{8^n} \int_{[-8,8]^n} (2^{-2k} + |\xi|^2)^{s/2} \wtv (\xi) e^{-\f{2\pi i}{16}\lan{\xi, l}} \, d\xi.
\end{align*}
Then {
\begin{align*}
&\Pi_3^1[f,g] (x)  \\
	&= \sum_{k\geq 0} 2^{2kn} \int_{\rn} \Phi^k_s (\xi+\eta) \vp(\xi)\wh{f}(2^{k}\xi) \wtv^k_{-s}(\eta)\vp(\eta)  \wh{J^s g}(2^{k}\eta) e^{2\pi i2^k\lan{\xi+\eta, x}} \, d\xi\, d\eta\\         
	&= \sum_{k\geq 0} 2^{2kn} \int_{\rn}\left( \sum_{m\in \bz^n} c_{m,k}^s e^{\f{2\pi i}{16}\lan{\xi+\eta, m}}\right) \chi_{[-8,8]^n}(\xi+\eta)\vp(\xi)\wh{f}(2^{k}\xi) \\
	&\hspace{70pt} \left(\sum_{l\in \bz^n} \wt{c}^{-s}_{l,k} e^{\f{2\pi i}{16}\lan{\eta, l}}\right) \chi_{[-8,8]^n}( \eta)
	\vp( \eta)  \wh{J^s g}(2^{k}\eta) e^{2\pi i2^k\lan{\xi+\eta, x}} \, d\xi\, d\eta \\
	&= \sum_{k\geq 0} 2^{2kn} \int_{\rn}\left( \sum_{m\in \bz^n} c_{m,k}^s e^{\f{2\pi i}{16}\lan{\xi+\eta, m}}\right) \vp(\xi)\wh{f}(2^{k}\xi) \\
	&\hspace{70pt} \left(\sum_{l\in \bz^n} \wt{c}^{-s}_{l,k} e^{\f{2\pi i}{16}\lan{\eta, l}}\right)   \vp( \eta)  \wh{J^s g}(2^{k}\eta) e^{2\pi i2^k\lan{\xi+\eta, x}} \, d\xi\, d\eta \, , 
\end{align*}}  
 since the characteristic functions are equal to one on the support of $\Psi(\xi)\Psi(\eta)$. 
Lemma~\ref{le:inhom} implies that $b_m^s := \sup_{k\geq 0} |c_{m,k}^s| = O(|m|^{-n-s})$.  The following lemma shows that $\wt{b}_{m}^{-s} := \sup_{k\geq 0} |\wt{c}_{m,k}^{-s}|$ also has a fast decay.

\begin{lemma}\label{le:decay}
Let $\{f_{\de}\}_{\de\in [0,1]}$ be a family of Schwartz functions such that $\wh{f_{\de}}$ is supported on a compact set $K\in \rn$ for all $\de\in [0,1]$.  If $\sup_{\de \in [0,1]}\n{\De^N \wh{f_{\de}}}{L^{\infty}_{\xi}} < B$, then there is some constant $C(N)$ such that
\[
\sup_{\de\in [0,1]}|f_{\de}| (x) \leq B |K| C(N) |x|^{-2N}.
\]
\end{lemma}

\textbf{Remark:} Let $\cf^{-1}$ denote the inverse Fourier transform, i.e., the Fourier transform composed with the reflection $x\mapsto -x$. Consider the family  $\{ f_{\de}\}_{\de\in [0,1]}$ defined by $f_{\de}:= (\de - \De)^{-\f{s}{2}} \cf^{-1}[\wtv]$.   Note that $\wh{f_{\de}}(\xi) = (\de+|\xi|^2)^{-\f{s}{2}} \wtv(\xi)$ is smooth for $(\de,\xi) \in [0,1]\times \rn$ and compactly supported in $\xi$.  Thus for any $\al \in \bz^n$, $\p^{\al} \wh{f_{\de}}$ is continuous and compactly supported, thus satisfying the condition of the Lemma above. Additionally, $f_{\de}$ is uniformly bounded for $(\de,x) \in [0,1]\times \rn$ as seen by the following.
\[
\n{(\de -\De)^{-\f{s}{2}}\cf^{-1}[\wtv]}{L^{\infty}_x} \leq \n{ (\de+ |\cdot|^2)^{-\f{s}{2}}\wtv}{L^1_{\xi}} \leq \n{|\cdot|^{-\f{s}{2}} \wtv}{L^1_{\xi}}
\]   
Thus, we obtain that $\wt{b}_j^{-s} \leq \sup_{\de \in [0,1]} |f_{\de}|(j) =  O((1+|j|)^{-N})$ for any $N\in \bn$.
\begin{proof}
Using the identity $\De^N e^{i\lan{\xi, x}} = C_N |x|^{2N}e^{i\lan{\xi ,x}}$, we apply   Green's theorem:
\begin{align*}
|x|^{2N} |f_{\de}| (x) &= \left|\int_{\rn} \wh{f_{\de}} (\xi) |x|^{2N}e^{2\pi i\lan{\xi, x}}\, d\xi\right|= C_N \left|\int_{\rn} \wh{f_{\de}} (\xi) \De^{N}e^{2\pi i\lan{\xi, x}}\, d\xi\right|\\
	&= C_N\left|\int_{K} \De^{N}\wh{f_{\de}} (\xi) e^{2\pi i\lan{\xi, x}}\, d\xi\right| \leq C_N |K| \sup_{\de\in [0,1]} \n{\De^N \wh{f_{\de}}}{L^{\infty}_{\xi}}.
\end{align*}
This proves Lemma \ref{le:decay}.
\end{proof}

{ We now continue the proof of Theorem~\ref{th:KP1} for the inhomogeneous case. We have } 
\begin{align*}
& \abs{\Pi_3^1[f,g]}(x) \\
	& = \left| \sum_{m,l\in \bz^n} \sum_{k\geq 0} c_{m,k}^s \wt{c}^{-s}_{l,k}   \right. \\
&\hspace{25pt}\left.\int_{\rn} e^{2\pi i \lan{2^{-k}\xi, \f{m}{16}}} \vp(2^{-k}\xi)\wh{f}(\xi)  e^{2\pi i\lan{2^{-k}\eta, 
\f{m+l}{16}}}\vp(2^{-k}\eta)  \wh{J^s g}(\eta) e^{2\pi i\lan{\xi+\eta,x}} \, d\xi\, d\eta\right| \\
	& \leq \sum_{m,l \in \bz^n} b_{k}^s \wt{b}_{k}^{-s}   \\
	 &\hspace{25pt}\sum_{k\geq 0} \left|\int_{\rn} e^{2\pi i\lan{2^{-k}\xi, \f{m}{16}}} \vp(2^{-k}\xi)\wh{f}(\xi) e^{2\pi i \lan{2^{-k}\eta,\f{m+l}{16}}}\vp( 2^{-k}\eta)  \wh{J^s g}(\eta) e^{2\pi i\lan{\xi+\eta, x}} \, d\xi\, d\eta\right|\\
	& = \sum_{m,l \in \bz^n} b_{k}^s \wt{b}_{k}^{-s} \sum_{k \geq 0}\left|[\De^m_k f](x) [\De^{m+l}_k J^s g](x) \right|, 
\end{align*} 
where $[\De_k^m f](\cdot) := \int_{\rn} 2^{kn} \vp(2^k (\cdot -y)+\f{m}{16}) f(y) \, dy$.

Letting $r_* := \min(r,1)$,
\begin{align*}
\n{\Pi_3^1[f,g]}{L^r}^{r_*}	&\leq \sum_{m,l\in \bz^n} |b_m^s \wt{b}_l^{-s}|^{r_*} \n{\sqrt{\sum_{k\geq 0} |\De^m_k f|^2(x)}\,  \sqrt{\sum_{k\geq 0} |\De^{m+l}_k J^s g|^2(x)} }{L^r}^{r_*}\\
	&= \sum_{m,l\in \bz^n}  |b_m^s \wt{b}_{ { l}}^{-s}|^{r_*} \n{ \sqrt{\sum_{k\geq 0} |\De^m_k f|^2}}{L^p(\rn)}^{r_*} \n{\sqrt{\sum_{k\geq 0} |\De^{m+l}_k J^s g|^2}}{L^q(\rn)}^{r_*}
\end{align*}
for any $\f{1}{p}+ \f{1}{q}= \f{1}{r}$.  Applying Corollary~\ref{cor:m} to the right hand side above gives the estimate for the sum $k\geq 0$, { since,  as observed, $b_m^s = O(|m|^{-n-s})$  and the series in $m$ and $l$ converge when $s>n/r-n$. }

 For $k<0$ define   the Fourier coefficient 
\[ 
a_m^s := { \f{1}{8^n}} \int_{[-8,8]^n} (1+|\xi|^2)^{s/2}\Phi(2\xi)e^{-\f{2\pi i}{16}\lan{\xi, m}}\, d\xi. 
\]
of the function $\Phi(2\xi)$, { which is $O( (1+|m|)^{-N})$ for all $N>0$.} Then we have  
\begin{align*}
&\Pi_3^2[f,g](x) \\
 :& = \sum_{k < 0} \int_{\rn} (1+|\xi+\eta|^2)^{\f{s}{2}} \vp(2^{-k} \xi)\wh{f}(\xi) \vp(2^{-k} \eta) \wh{g}(\eta) e^{2\pi i\lan{\xi+\eta, x}} \, d\xi\, d\eta\\
	&= \sum_{k < 0} \int_{\rn} (1+|\xi+\eta|^2)^{\f{s}{2}} \Phi(2(\xi+\eta)) \vp(2^{-k} \xi)\wh{f}(\xi)  \\
	&\hspace{100pt}(1+|\eta|^2)^{-\f{s}{2}}\Phi(2\eta)\vp(2^{-k} \eta)  \wh{J^s g}(\eta) e^{2\pi i\lan{\xi+\eta, x}} \, d\xi\, d\eta\\
	&= \sum_{k < 0} \int_{\rn} \left(\sum_{m\in \bz^n} a_m^s e^{\f{2\pi i}{16}\lan{\xi+\eta, m}}\right) \chi_{[-8,8]^n} (\xi+\eta) \vp(2^{-k} \xi)\wh{f}(\xi) \\
	&\hspace{50pt} \left(\sum_{l\in \bz^n} a_l^{-s} e^{\f{2\pi i}{16}\lan{\eta, l}} \right)\chi_{[-8,8]^n} (\eta) \vp(2^{-k} \eta) \wh{J^s g}(\eta) e^{2\pi i\lan{\xi+\eta, x}} \, d\xi\, d\eta \\
	&= \sum_{k < 0} \int_{\rn} \left(\sum_{m\in \bz^n} a_m^s e^{\f{2\pi i}{16}\lan{\xi+\eta, m}}\right) \vp(2^{-k} \xi)\wh{f}(\xi) \\
	&\hspace{100pt} \left(\sum_{l\in \bz^n} a_l^{-s} e^{\f{2\pi i}{16}\lan{\eta, l}} \right)\vp(2^{-k} \eta) \wh{J^s g}(\eta) e^{2\pi i\lan{\xi+\eta, x}} \, d\xi\, d\eta	 \, , 
\end{align*}
{ since the characteristic functions are equal to $1$ on the support of 
$\vp(2^{-k} \xi)\vp(2^{-k} \eta)$, since $k<0$. }
Note that $c_m^s$ and $c_m^{-s}$ are $O(|m|^{-N})$ for any $N\in \bn$.  We conclude that 
\begin{align*}
&\Pi_3^2[f,g](x) \\
	&= \sum_{m,l\in \bz^n} a_m^s a_l^{-s}   \\
	&\hspace{26pt}\sum_{k < 0} \int_{\rn}  \vp(2^{-k} \xi)e^{2\pi i \lan{\xi,\f{m}{16}}}\wh{f}(\xi) \vp(2^{-k} \eta) e^{2\pi i \lan{\eta, \f{m+l}{16}}} \wh{J^s g}(\eta) e^{2\pi i\lan{\xi+\eta, x}} \, d\xi\, d\eta\\
	&= \sum_{m,l\in \bz^n} a_m^s a_l^{-s} \sum_{k < 0} \int_{\rn}  \vp(2^{-k} \xi)\wh{\tau_{\f{m}{16}} f}(\xi) \vp(2^{-k} \eta) \wh{\tau_{\f{m+l}{16}} J^s g}(\eta) e^{2\pi i\lan{\xi+\eta, x}} \, d\xi\, d\eta\\
	&= \sum_{m,l\in \bz^n} a_m^s a_l^{-s} \sum_{k<0} [\De_k \tau_{\f{m}{16}} f](x) [\De_k \tau_{\f{m+l}{16}}J^s g] (x)\, ,
\end{align*}
where $\tau_m$ is the translation operator $[\tau_m f](x) = f(x-m)$.  Taking the $L^r$ norm, we obtain 
\begin{align*}
\n{\Pi_3^2 [f,g]}{L^r}^{r_*} &\leq \sum_{m,l \in \bz^n} |a_m^sa_l^{-s}|^{r_*} \n{ \sqrt{ \sum_{k\in \bz} |\De_k \tau_m f|^2}}{L^p}^{r_*} \n{\sqrt{ \sum_{k\in \bz} |\De_k \tau_{m+l}J^s g|^2}}{L^q}^{r_*}\\
&\leq \sum_{m,l \in \bz^n} |a_m^s a_l^{-s}|^{r_*} \n{ f}{L^p}^{r_*} \n{ J^s g }{L^q}^{r_*}.
\end{align*}
In view of the rapid decay
 of $a_m^s$ and $a_m^{-s}$, we conclude  the proof of Case 1.

\textbf{Case 2:} $1<r<\infty$, $(p, q) \in \{(r,\infty), (\infty,r)\}$

We again adapt the proof given in \cite{BB}.    Following the computations in Section~\ref{sec:homKP}, for $j\geq 2$, $[\De_j \Pi_3^1[f,g]](x)$ can be written as	
\begin{align*}
&\sum_{k\geq j-2} \int_{\br^{2n}}2^{js} \Psi_{j,s} (2^{-j} (\xi+\eta))  \Psi(2^{-k} \xi) \wh{f}(\xi)  2^{-ks} \Psi_{k,-s}(2^{-k} \eta) \wh{J^s g}(\eta)\, e^{2\pi i\lan{\xi+\eta, x}}\, d\xi\, d\eta\\
	&\leq 2^{js} \left(\sum_{k\geq j-2} 2^{-2ks}\right)^{\f{1}{2}} \left( \sum_{k\geq j-2} \abs{\De_{j,s} \Big[ [\De_k f] [\De_{k,-s} J^s g]\Big](x)}^2\right)^{\f{1}{2}}
\end{align*}
where $\Psi_{j,s}(\cdot) := (2^{-2j} +|\cdot|^2)^{s/2} \Psi(\cdot)$; and the operators $\cf[\De_{j,s} f]:= \Psi_{j,s}(2^{-j} \cdot) \wh{f}(\cdot)$.  Note that the family $\{\De_{j,s}\}_{j\geq 0}$ is not a Littlewood-Paley { family} in the usual sense, i.e. it is not given by convolution with $L^1$~dilations of a single kernel.  Rather, it is given by convolution with   kernels that  are different for each $j\geq 0$.  Below, we will show that $\{\De_{j,s}\}_{j\geq 0}: L^p \to L^p\ell^2$ for $1<p<\infty$ and $s\in \br$.

When $j<2$, $\big| [\De_j \Pi_3^1[f,g]] \big| (x)$ can be written as
\begin{align*}
	& \left| \sum_{k\geq 0} \int_{\br^{2n}}  \Phi_s (\xi+\eta) \Psi (2^{-j} (\xi+\eta))  \Psi(2^{-k} \xi) \wh{f}(\xi)  2^{-ks} \Psi_{k,-s} (2^{-k} \eta) \wh{J^s g}(\eta)\, e^{2\pi i\lan{\xi+\eta, x}}\, d\xi\, d\eta \right| \\
	&\leq \left(\sum_{k\geq 0} 2^{-2ks}\right)^{\f{1}{2}} \left( \sum_{k\geq 0} \abs{S_2^s \De_j \left[ [\De_k f] [\De_{k,-s} J^s g]\right](x)}^2\right)^{\f{1}{2}}, 
\end{align*}
where $\Phi_{s}(\cdot) := (1 +|\cdot|^2)^{s/2} \Phi(2^2 \cdot)$ and $ S_2^{s}  f = \cf^{-1}[ \Phi_s \wh{f} \,]$. Then,

\begin{align*}
\n{\Pi_3^1[f,g]}{L^r} &\leq C(r,n,s)\left(\n{ \left(\sum_{j\geq 2} \sum_{k\geq j-2} \abs{\De_{j,s} [ \De_k f\,\De_{k,-s} J^s g]}^2\right)^{\f{1}{2}}}{L^r}\right.\\
 &\left.  \qq\qq +\n{ \left(\sum_{j<2} \sum_{k\geq 0} \left|S^s_2 \De_j [ \De_k f  \, \De_{k,-s} J^s g] \right|^2\right)^{\f{1}{2}}}{L^r}\right).
\end{align*}
The operator $\{S^s_2 \De_j\}_{j\in \bz} = \{\De_j S^s_2\}_{j\in \bz} : L^r \to L^r \ell^2$ is clearly bounded for $1<r<\infty$.   Next, we will show that $\{\De_{j,s}\}_{j\geq 0}: L^r \to L^r \ell^2$.   Recall that we have introduced above $\wtv\in \cs(\rn)$ supported on an slightly larger annulus than that of $\Psi$ such that  $\Psi = \wt{\Psi} \Psi$.  We write
\[
[\wt{\De_j} u](x) = \int_{\rn}\wt{\Psi_j^s}(2^{-j}\xi) \Psi (2^{-j} \xi) \wh{u}(\xi)e^{2\pi i\lan{\xi, x}}\, d\xi
\]
where $\wt{\Psi_j^s}(\cdot) :=  (2^{-2j} + |2^{-j} \cdot |^2)^{\f{s}{2}}\wtv (\cdot)$.
Expanding $\wt{\Psi_j^s}$ in Fourier series with coefficients denoted $\wt{c_{m,j}^s}$, we can write
$\wt{\De_j} u= \sum_{m\in \bz^n} \wt{c_{m,j}^s} \De_j^{m} u$ where
\[
[\De_j^m u](x) = \int_{\rn} e^{\f{2\pi i}{16}\lan{2^{-j} \xi, m}}\vp (2^{-j} \xi) \wh{u}(\xi) e^{2\pi i\lan{\xi,x}}\, d\xi.
\]  
Defining $\wt{b_m^s} := \sup_{j\geq 0} |\wt{c_{m,j}^s}|$, we recall from Lemma~\ref{le:decay} that $\wt{b_m^s} = O((1+|m|)^{-N})$ for any $N\in \bn$.  Thus, applying Corollary~\ref{cor:m},
\begin{align*}
\n{\left(\sum_{j\geq 0} \abs{\wt{\De_j} u}^2\right)^{\f{1}{2}}}{L^r} &\leq \n{\left(\sum_{j\geq 0} \abs{ \sum_{m\in \bz^n} \wt{b}_{m}^s |\De_j^{m} u|}^2\right)^{\f{1}{2}}}{L^r} \leq \sum_{m\in \bz^n} \wt{b}_m^s \n{\left(\sum_{j\geq 0} \abs{\De_j^m u}^2\right)^{\f{1}{2}}}{L^r} \\
&\leq C(n,r) \sum_{m\in \bz^n} \wt{b}_m^s \ln (1+ |m|) \n{u}{L^r}.
\end{align*}
Using \cite[Proposition 4.6.4]{G1}, we extend the operators $\{S^s_2 \De_j\}_{j\in \bz}$ and $\{\De_{j,s}\}_{j\geq 0}$ from $L^r \to L^r \ell^2$  to $L^r \ell^2 \to L^r \ell^2 \ell^2$ for $1<r<\infty$ and obtain
\begin{align*}
\n{\Pi_3^1[f,g]}{L^r} &\leq C(n,r,s) \n{ \left(\sum_{k\geq 0} \abs{ \De_k f\,\De_{k,-s} J^s g}^2\right)^{\f{1}{2}}}{L^r}\\
	&\leq C(n,r,s) \sup_{k\geq 0} \n{\De_{k,-s} J^s g}{L^\infty} \n{f}{L^r}\\
	&\leq \n{f}{L^r}\n{J^s g}{L^{\infty}}\sup_{k\geq 0} \Big\|{\wh{\wt{\vp_{k}^{-s}}}}\Big\|_{L^1} 
\end{align*}
for $1<r<\infty$.  Applying Lemma~\ref{le:decay} and the remark following, we obtain  that 
$$
\sup_{k\geq 0} \Big| \wh{\wt{\vp_k^{-s}}}  \Big|(x) = \sup_{\geq 0} \left|\cf \left[(2^{-2k} - |\cdot|^2)^{-s/2} \wt{\vp}(\cdot)\right] \right|(x) = O((1+|x|)^{-N})
$$
 for any $N\in \bn$.  Taking $N>n$ gives the necessary estimate for $\Pi_3^1[f,g]$.

It remains to obtain the endpoint estimates for $\Pi_3^2[f,g]$. For this, we write
\[
\Pi_3^2[f,g] = S^s_2 \left[ \sum_{k\leq 0} (\De_k f )\,  (\De_k S^{-s}_2 J^s g)\right].
\]
Note that, for any $s\in \br$, $S^s_2$ is a $L^p$~multiplier for $1\leq p\leq \infty$ since it is a convolution with an $L^1(\rn)$ function.  Also, the symbol $\sum_{k<0} \vp(2^{-k} \xi) \vp (2^{-k} \eta)$ satisfies the Coifman-Meyer condition in Theorem~\ref{th:CM}.  Thus we obtain
\[
\n{\Pi_3^2[f,g]}{L^r} \leq C(n,r,s) \n{f}{L^p} \n{S^{-s}_2 J^s g}{L^q} \leq C(n,r,s) \n{f}{L^p} \n{J^s g}{L^q}
\]
for any $1\leq r<\infty$, $1\leq p,q\leq \infty$ with $\f{1}{p} + \f{1}{q} =\f{1}{r}$.
\end{proof}

\section{Multi-parameter Kato-Ponce inequality}\label{sec:multiKP}
Let $f,g\in \cs(\rn)$, we want to prove the multi-parameter Kato-Ponce inequality.  Write $\rn = \br^{n_1} \times \br^{n_2} \times \dots \times \br^{n_d}$ and denote for $x\in \rn$, $x= (x_1,x_2,\dots, x_d)$ where $x_j \in \br^{n_j}$ for $j=1,2,\dots ,d$.  

For $s\in \br$, define fractional partial derivatives $|D_{x_j}|^s$ by $\cf [|D_{x_j}|^{s} f] (\xi) := |\xi_j|^{s} \wh{f}(\xi)$.  Let  $E:= \{1,2,\dots, d\}$ and $\mathcal{P}[E]$ be its power set.  For $B \in \mathcal{P}[E]$, denote $D_{x(B)}^{s(B)} := \prod_{j\in B} D_{x_j}^{s_j}$.  Then the multi-parameter homegeneous Kato-Ponce inequality can be stated as follows.  

\begin{theorem}\label{th:multiKP1}
Given $s_j > \max(n_j/r - n_j,0)$ for $j= 1,2,\dots, d$  there exists $C = C( d, n_j ,r, s_j, p(B),q(B))$ so that for all $f,g\in \cs(\rn)$,
\begin{equation}\label{eq:multiKP}
\n{D_{x(E)}^{s(E)} [fg]}{L^r(\rn)} \leq C \sum_{B\in \mathcal{P}[E]} \n{D_{x(B)}^{s(B)} f}{L^{p(B)}(\rn)} \n{D_{x(E \setminus B)}^{s(E \setminus B)} g}{L^{q(B)}(\rn)}
\end{equation}
for $\f{1}{2}<r<\infty$, $1<p(B),q(B)\leq \infty$ satisfying $\f{1}{p(B)} + \f{1}{q(B)} = \f{1}{r}$.
\end{theorem}

\begin{theorem}\label{th:multiKP2}
When $s_j \leq \max(n_j/r - n_j,0)$ for some $j= 1,2,\dots, d$, \eqref{eq:multiKP} fails.
\end{theorem}

\begin{proof}[Proof of Theorem~\ref{th:multiKP2}]
Let $f^{(n_j)}\in \cs(\br^{n_j})$ be non-zero functions for $j=1,\dots, d$.  Then define $F (x):= \prod_{j=1}^d F^{(n_j)} ( x_j)$. Then $\wh{F}(\xi) = \prod_{j=1}^d \wh{f^{(n_j)}}(\xi_j)$ and $[D_{x(E)}^{s(E)} F](x) = \prod_{j=1}^{d} [D^{s_j} f^{(n_j)}](x_j)$.  Thus their $L^p(\rn)$ norms split into a product of $L^p(\br^{n_j})$ norms.  Thus letting $f=F$ and $g=\overline{F}$ in \eqref{eq:multiKP}, Lemma~\ref{le:homKP} gives that the left hand side is infinite if $s \leq \f{n_j}{r} - n_j$ for any $j= 1,\dots, d$, while the right hand side is finite as long as $p(B),q(B)>1$.

The argument for $s_j<0$ easily follows by a similar argument as in Section~\ref{sec:homKP}.
\end{proof}

Next we will prove Theorem~\ref{th:multiKP1}. First we make a few remarks below.  
\begin{itemize}
\item In the case of $\mathbf{R}^2$, Theorem~\ref{th:multiKP1} stated in the appendix of \cite{Acta}.  However, in view of Theorem~\ref{th:multiKP2}, we note that the inequality \cite[Equation (61)]{Acta}  holds only when $\min(\al, \be)  \, >  \,  \f{1}{r} - 1$ for $r<1$.  \emph{This point has been corrected in \cite{MS2}.}

\item The weak $L^r$ endpoints for these estimates could be false due  to the fact that $L^{r,\infty}$~norms cannot be iterated, i.e. $\n{f}{L^{r,\infty}(\br^2)} \neq \n{\n{f}{L^{r,\infty}(\br)}}{L^{r,\infty}(\br)}$.

\item From the proof given in Section~\ref{sec:inhKP} and the proof to be presented below, it will be apparent that the operators $D^s_{x_j}$ can be replaced by $J^s_{x_j}$ defined similarly.  We have not included this generalization in order to simplify the argument.
\end{itemize}
The proof of Theorem~\ref{th:multiKP1} is an iteration of the proof of Theorem~\ref{th:KP1} in Section~\ref{sec:homKP}, using multi-parameter Littlewood-Paley decompositions.  We introduce the corresponding operators here.

Let $\Phi^{(j)}\in \cs(\br^{n_j})$ be such that $\Phi \equiv 1$ when $|\xi|\leq 1$ and is supported on $|\xi|\leq 2$.  Define $\vp^{(j)}(\cdot) := \Phi^{(j)} (\cdot) - \Phi^{(j)}(2\cdot)$.  For technical reasons, we also define $\Psi^{[j]}\in \cs(\br^{n_j})$ to be supported on an annulus, and satisfying $\sum_{k\in \bz} |\Psi^{[j]} (2^{-k} \xi_j)|^2 =1$ for all $\xi_j \in \br^{n_j}\setminus\{0\}$. 

Define the operator $S_k^{(j)}$ by $\cf[S_k^{(j)} f](\xi) = \Phi^{(j)}(2^{-k}\xi_j)\wh{f}(\xi)$; $\De_k^{(j)}$ by $\cf[\De_k^{(j)} f](\xi) = \vp^{(j)}(2^{-k}\xi_j)\wh{f}(\xi)$; and
$\De_k^{[j]}$ by $\cf[\De_k^{[j]} f](\xi) = \Psi^{[j]}(2^{-k} \xi_j)\wh{f}(\xi)$.   Given $B\subset \{1,2,\dots, d\}$, define $S_{k(B)}^{(B)}$, $\De_{k(B)}^{(B)}$, $\De_{k(B)}^{[B]}$ by $\prod_{j\in B} S_{k_j}^{(j)}$, $\prod_{j\in B} \De_{k_j}^{(j)}$, $\prod_{j\in B} \De_{k_j}^{[j]}$ respectively.

The following lemma shows the boundedness of the corresponding square-functions in $L^p(\rn)$ for $1<p<\infty$.  This is a slight generalization of \cite[Theorem 5.1.6]{G1}.
\begin{lemma}\label{le:square}
Let $\wt{\Psi^{(j)}}\in \cs(\br^{n_j})$ satisfy the conditions \eqref{5.1.3-1aa} and \eqref{5.1.3-1a} in Theorem~\ref{th:LP} with the constant $B_j^2$, $B^2$ respectively for $j=1,2,\dots, d$. Let define $\wt{\De_{k}^{(j)}}$ by $\cf\left[\wt{\De_k^{(j)}} u\right](\xi) = \wt{\Psi^{(n_j)}} ( 2^{-k} \xi_j) \wh{u}(\xi)$.  Then for all $u\in \cs(\rn)$, there exists $C_j= C(n_j,p)<\infty$ satisfying
\[
\n{\sqrt{ \sum_{k_1, \dots, k_d \in \bz} \left|\wt{\De^{(1)}_{k_1}}\cdots \wt{\De^{(d)}_{k_d}} u\right|^2}}{L^p(\rn)} \leq \left[\prod_{j=1}^d  C_j B_j \max (p,(p-1)^{-1})\right] \n{u}{L^p(\rn)}.
\]
\end{lemma}
The proof of the lemma above is simply an iteration of Theorem~\ref{th:LP} $d$~times whilst commuting the $L^p(\br^{n_j})$ norms.  We refer to the proof given in \cite[Theorem 5.1.6]{G1} for this calculations.

The following lemma is due to Ruan, \cite[Theorem 3.2]{HP}.
\begin{lemma}\label{le:Hp}
Let $0<p<\infty$.  For all $u \in \cs(\rn)$, there exists $C= C(n,p)$ satisfying
\[
\n{u}{L^p(\rn)} \leq C \n{\sqrt{\sum_{k_1, \dots, k_d\in \bz} \left|\De^{[1]}_{k_1}\cdots \De^{[d]}_{k_d} u\right|^2}}{L^p(\rn)}.
\]
\end{lemma}
For $1<p<\infty$, this immediately follows from duality and Lemma~\ref{le:square}.  However, for $0<p\leq 1$, this is a consequence of a multi-parameter square-function characterization of Hardy spaces~$H^p(\rn)$.  We refer to \cite{HP} for details.

\begin{proof}[Proof of Theorem~\ref{th:multiKP1}]
We introduce   notation  to aid the computations.  Although we   strive to use clear and accurate notation  throughout, it will be inevitable at times to be flexible for the sake of  exposition. We have

\begin{align*}
\left[D_{x(E)}^{s(E)} [fg]\right]&(x) = \iint_{\br^{2n}} \left[\prod_{j=1}^d |\xi_j+\eta_j|^{s_j}\right] \wh{f}(\xi) \wh{g}(\eta) e^{2\pi i\lan{\xi+\eta, x}}\, d\xi\, d\eta\\
	&\hspace{-20pt}= \iint_{\br^{2n}} \prod_{j=1}^d\left[ |\xi_j+\eta_j|^{s_j} \sum_{k,l\in \bz} \vp^{(j)}(2^{-k} \xi_j) \vp^{(j)}(2^{-l} \eta_j)\right] \wh{f}(\xi)  \wh{g}(\eta) e^{2\pi i\lan{\xi+\eta, x}}\, d\xi\, d\eta . 
\end{align*}
For each $j=1,2,\dots d$, split the sum inside the square bracket   into $\sum_{l<k-1} \cdot + \sum_{k<l-1}\cdot + \sum_{|k-l|\leq 1}\cdot =: M^1_j  + M^2_j + M^3_j$ and   
take the product over $j=1,2,\dots d$, to obtain
\[
\prod_{j=1}^d [ M^1_j + M^2_j + M^3_j ](\xi_j,\eta_j) = \sum_{a_1=1}^3\sum_{a_2=1}^3  \dots\sum_{a_d=1}^3 \left[ \prod_{j=1}^d M^{a_j}_j(\xi_j,\eta_j) \right].
\]
Define $J := [\bz / 3\bz]^d$ to be the set of $n$-tuples where each component is from $\{1,2,3\}$.  Given $A = (a_1, a_2,\dots, a_n) \in J$, define $M_A(\xi,\eta) := \prod_{j=1}^d M_j^{a_j}(\xi_j,\eta_j)$ so that the sum above can be expressed as $\sum_{A \in J} M_A(\xi,\eta)$.  Denoting 
\[
\Pi_A[f,g] (x):= \int_{\br^{2n}} M_A(\xi,\eta) \wh{f}(\xi) \wh{g}(\eta) e^{2\pi i\lan{\xi+\eta, x}}\, d\xi \, d\eta,
\]
we have $D^{s(E)}_{x(E)} [fg]= \sum_{A\in J} \Pi_A [f,g]$.  Thus, given $A\in J$, it suffices to show \eqref{eq:multiKP} for $\Pi_A[f,g]$.

Fix $A = (a_1, a_2, \dots, a_d) \in J$.  We define the sets $A_1, A_2, A_3$ such that for $\al = 1,2,3$, $A_{\al} := \{j\in E: a_j = \al\}$.  For any $A\in J$, $\{A_1, A_2, A_3\}$ forms a partition of $E$. For $\al,\be \in \{1,2,3\}$, $A_{\al,\be} := A_{\al} \cup A_{\be}$.  Roughly speaking, $A_{1,2}$ represents the components which have the high-low frequency interactions, and $A_3$ represents the ones with high-high interactions.  

Iterate the $L^r(\rn)$~norm by $\n{\Pi_A [f,g]}{L^r(\rn)} = \n{ \n{ \Pi_A[f,g]}{L^r_{A_{1,2}}\left(\br^{|A_{1,2}|}\right)}  }{L^r_{A_3}\left(\br^{|A_3|}\right)}$.  For the norm inside, we apply Lemma~\ref{le:Hp} to obtain
\[
\n{\Pi_A[f,g]}{L^r_{A_{1,2}}\left(\br^{|A_{1,2}|}\right)} \leq C(n,r) \n{ \sqrt{\sum_{\tiny m_j\in \bz;\, j\in A_{1,2}} \abs{ \De^{[A_{1,2}]}_{m(A_{1,2})}  \Pi_A [f,g] }^2} }{L^r_{A_{1,2}}\left(\br^{|A_{1,2}|}\right)}.
\]
For $j\in E$ and $\al=1,2$, denote $\wt{M_{j,m}^{\al}}:= \psi^{[j]}(2^{-m}(\xi+\eta)) M^{\al}_j(\xi,\eta)$, and $\wt{M_{j,m}^{3}} := M_j^3$.  Then
\[
\left[\De^{[A_{1,2}]}_{m(A_{1,2})}  \Pi_A [f,g]\right](x) = \int_{\br^{2n}} \left[\prod_{j=1}^d \wt{M^{a_j}_{j,m_j}}(\xi_j,\eta_j) \right] \wh{f}(\xi) \wh{g}(\eta) e^{2\pi i\lan{\xi+\eta, x}}\, d\xi\, d\eta.
\]
We analyze the symbols $\wt{M^{a_j}_{j,m_j}}$ and transfer the fractional derivatives onto the high-frequency term.  
\begin{align*}
\wt{M^1_{j,m}} &= |\xi_j+\eta_j|^{s_j} \vp^{[j]}(2^{-m}(\xi_j + \eta_j))\sum_{k\in \bz} \vp^{(j)}(2^{-k} \xi_j) \Phi^{(j)}(2^{-k+2} \eta_j)\\
	&= |\xi_j+\eta_j|^{s_j} \sum_{|k-m|\leq 3} \vp^{[j]}(2^{-m}(\xi_j + \eta_j))\vp^{(j)}(2^{-k} \xi_j) \Phi^{(j)}(2^{-k+2} \eta_j)\\
	&= \sum_{|k-m|\leq 3}2^{(m-k)s} \vp_{s_j}^{[j]}(2^{-m}(\xi_j + \eta_j))  \vp_{-s_j}^{(j)}(2^{-k} \xi_j)\Phi(2^{-k+2} \eta_j) |\xi_j|^{s_j}, 
\end{align*}
where $\vp_{s_j}^{[j]} (\cdot):= |\cdot|^{s_j} \vp^{[j]} (\cdot)$ and $\vp^{(j)}_{-s_j} (\cdot):= |\cdot|^{-s_j} \vp^{(j)}(\cdot)$.  Note that $\vp^{(j)}_{-s_j} \in \cs(\br^{n_j})$, so that the slight change in the operator $\De_{m}^{(j)}$ is not significant.  Thus we will ignore this difference.  Also, we will replace the finite sum above by a larger constant in the end.  Expanding $\vp_{s_j}^{[j]}$ into its Fourier series, we obtain
\begin{align*}
\wt{M^1_{j,m}}(\xi_j,\eta_j) &= \sum_{l\in \bz^{n_j}} \wt{c_l^{s_j}} e^{\f{2\pi i}{16} \lan{ 2^{-k}(\xi_j +\eta_j),l}} \vp_{-s_j}^{(j)}(2^{-k} \xi_j)\Phi^{(j)}(2^{-k+2} \eta_j) \xi_{[-8,8]^n}(2^{-k}(\xi+\eta)) |\xi_j|^{s_j}\\
	&= \sum_{l\in \bz^{n_j}} \wt{c_l^{s_j}} e^{\f{2\pi i}{16} \lan{ 2^{-k} \xi_j, l }} \vp_{-s_j}^{(j)} (2^{-k} \xi_j)e^{\f{2\pi i}{16}\lan{ 2^{-k}\eta_j, l}} \Phi^{(j)}(2^{-k+2} \eta_j) |\xi_j|^{s_j}
\end{align*}
where $\wt{c_l^{s_j}}:= 8^{-n} \int_{[-8,8]^{n_j}} |\xi_j|^s \vp^{(n_j)} (\xi_j) e^{-\f{2\pi i}{16}\lan{\xi_j, l}}\, d\xi_j$.  Similarly,
\[
\wt{M^2_{j,m}}(\xi_j,\eta_j) = \sum_{l\in \bz^{n_j}} \wt{c_l^{s_j}} e^{\f{2\pi i}{16} \lan{ 2^{-k} \xi_j, l }}\Phi^{(j)}(2^{-k+2} \xi_j) e^{\f{2\pi i}{16}\lan{ 2^{-k}\eta_j, l}}\vp_{-s_j}^{(j)}(2^{-k} \eta_j) |\eta_j|^{s_j}.
\]
For $\wt{M_{j,m}^{3}} = M_j^3$, 
\[
M^3_j(\xi_j,\eta_j) = \sum_{l\in \bz^{n_j}} c_l^{s_j} \sum_{k\in \bz} e^{\f{2\pi i}{16}\lan{ 2^{-k} \xi_j,l} } \vp(2^{-k} \xi_j) e^{\f{2\pi i}{16}\lan{ 2^{-k} \eta_j, l} } \vp^{(j)}_{-s_j} (2^{-k} \eta_j) |\eta_j|^{s_j}
\]
where $c_l^{s_j} := 8^{-n}\int_{[-8,8]^{n_j}} |\xi_j|^{s_j} \Phi^{(j)}(\xi_j) e^{-\f{2\pi i}{16} \lan{\xi, l}}\, d\xi_j$ by Lemma~\ref{le:homKP}.  Recall that $c_l^{s_j}, \wt{c_l^{s_j}} = O((1+|l|)^{-n_j - s_j})$.  As in the previous sections, we can pull the summation in $l_j$ for $j=1,2,\dots,d$ outside the norm.  Also, the shift $e^{\f{2\pi i}{16}\lan{2^{-k}\xi_j, l}}$ acting on the Littlewood-Paley operators creates a logarithmic term via Lemma~\ref{le:square}, which can be controlled due to the fast decay of $|c_l^{s_j}|^{r_*}$ and $|\wt{c_l^{s_j}}|^{r_*}$ where $r_* = \min (r,1)$.  Thus, we   ignore the summation in $l$ and the shift operators $e^{\f{2\pi i}{16}\lan{2^{-k}\cdot, l}}$ acting on $\wh{\De_k^{(j)}}$.

Therefore, we can reduce the key expression as follows:
\[
\De_{m(A_{1,2})}^{[A_{1,2}]} \Pi_A [f,g] \approx \sum_{\tiny \begin{array}{c} k_j\in \bz;\\ j\in A_3\end{array}} \left[ \De_{m(A_{1})}^{(A_{1})} S_{m(A_2)}^{(A_2)} \De_{k(A_3)}^{(A_3)} D^{s(A_1)}_{x(A_1)} f\right] \left[ S_{k(A_1)}^{(A_1)} \De_{m(A_2)}^{(A_2)} \De_{k(A_3)}^{(A_3)} D^{s(A_{1,3})}_{x(A_{1,3})} g\right].
\]
Now applying the Cauchy-Schwarz inequality  for the summation in $k_j: j\in A_3$, and the $\ell^1-\ell^{\infty}$~H\"older inequality for $k_j: j\in A_{1,2}$, we obtain
\begin{align*}
&\n{\Pi_A[f,g]}{L^r(\rn)} \lesssim_{n,r} \n{ \left(\sum_{\tiny \begin{array}{c} m_j\in \bz;\\  j\in A_{1,2}\end{array}} \abs{ \De^{[A_{1,2}]}_{m(A_{1,2})}  \Pi_A [f,g] }^2\right)^{\f{1}{2}} }{L^r(\rn)}\\
	&\lesssim \n{ \left(\sum_{\tiny\begin{array}{c} k_{j}\in \bz\\ j\in A_{1,3} \end{array}}  \abs{M^{A_2} \De^{(A_{1,3})}_{k(A_{1,3})} D_{x(A_1)}^{s(A_1)} f}^2\right)^{\f{1}{2}} }{L^p} \n{ \left(\sum_{\tiny\begin{array}{c} k_{j}\in \bz\\ j\in A_{2,3} \end{array}} \abs{M^{A_1} \De_{k(A_{2,3})}^{(A_{2,3})} D_{x(A_{2,3})}^{s(A_{2,3})} g }^2 \right)^{\f{1}{2}}}{L^q}
\end{align*}
for $\f{1}{p} + \f{1}{q} =\f{1}{r}$ and $1<p,q<\infty$, where $M^{A_2}$ and $M^{A_1}$ represents the appropriate Hardy-Littlewood maximal functions.
We apply Fefferman-Stein's inequality \cite{FS} to remove the maximal functions.  Then the quantity above is controlled by
\[
\n{\left(\sum_{\tiny\begin{array}{c} k_{j}\in \bz\\ j\in A_{1,3} \end{array}}  \abs{\De^{(A_{1,3})}_{k(A_{1,3})} D_{x(A_1)}^{s(A_1)} f}^2\right)^{\f{1}{2}}}{L^p(\rn)} \n{\left(\sum_{\tiny\begin{array}{c} k_{j}\in \bz\\ j\in A_{2,3} \end{array}} \abs{ \De_{k(A_{2,3})}^{(A_{2,3})} D_{A_{2,3}}^{s ( A_{2,3} )  } g }^2 \right)^{\f{1}{2}}}{L^q(\rn)}.
\]
Commuting the $L^r(\br^{n_j})$ norms appropriately so that we can apply Lemma~\ref{le:square}, this quantity is bounded by 
\[
\n{D_{x(A_1)}^{s(A_1)} f}{L^p(\rn)} \n{D_{x(A_{2,3})}^{s(A_{2,3})} g}{L^q(\rn)}.
\]
This concludes the proof in the cases $\f{1}{2}<r<\infty$, $1<p,q<\infty$.

Consider the endpoint case $1<r<\infty$, $p=r$ and $q=\infty$.  We begin by applying Lemma~\ref{le:Hp} and following the computations from Section~\ref{sec:homKP}.
\begin{align*}
&\n{\Pi_A[f,g]}{L^r(\rn)} \lesssim_{n,r} \n{ \sqrt{ \sum_{\tiny m_j \in \bz; j\in E} \abs{ \De_{m(E)}^{[E]} \Pi_A [f,g]}^2}}{L^r(\rn)}\\
	&\lesssim \n{\sqrt{\sum_{\tiny m_j, k_j \in \bz; j\in E} \abs{ \wt{\De_{m(E)}^{[E]}} \left[ [\De_{k(A_{1,3})}^{(A_{1,3})} S_{k(A_2)}^{(A_2)} D_{x(A_1)}^{s(A_1)} f] [S_{k(A_1)}^{(A_1)}\De_{k(A_{2,3})}^{(A_{2,3})}  D_{x(A_{2,3})}^{s(A_{2,3})} g]\right]}^2}}{L^r(\rn)}
\end{align*}
where $\wt{\De_{m(E)}^{[E]}}$ is a bounded operator mapping $L^p \to L^p \ell^2$ due to Lemma~\ref{le:square}, so that we can apply \cite[Proposition 4.6.4]{G1}.  Noting that $\sup_{k \in \bz} S_{k(A_1)}^{(A_1)}\De_{k(A_{2,3})}^{(A_{2,3})}: L^{\infty}(\rn) \to L^{\infty}(\rn)$ is a bounded operator, the endpoint estimates follow as before.

This concludes the proof of Theorem~\ref{th:multiKP1}. 
\end{proof}

\end{document}